\documentclass[oneside,11pt]{amsart}
\usepackage{amssymb, amscd}
\usepackage[all]{xy}
\usepackage{enumerate}

\newtheorem{theorem}{Theorem}[section]
\newtheorem{definition}[theorem]{Definition}
\newtheorem{lemma}[theorem]{Lemma}

\newtheorem{proposition}[theorem]{Proposition}
\newtheorem{corollary}[theorem]{Corollary}

\newtheorem{remark}[theorem]{Remark}
\newtheorem{example}[theorem]{Example}

\author{Dieter Degrijse}
\address{Department of Mathematics, Catholic University of Leuven, Kortrijk, Belgium}
\email{Dieter.Degrijse@kuleuven-kortrijk.be}

\title{On a filtration of the second cohomology of nilpotent Lie algebras}
\date{\today}

\thanks{The author was supported by the Research Fund K.U.Leuven.}
\begin{document}
\begin{abstract}
\noindent
We study a known filtration of the second cohomology of a finite dimensional nilpotent Lie algebra $\mathfrak{g}$ with coefficients in a finite dimensional nilpotent $\mathfrak{g}$-module $M$, that is based upon a refinement of the correspondence between $\mathrm{H}^2(\mathfrak{g},M)$ and equivalence classes of abelian extensions of $\mathfrak{g}$ by $M$. We give a different characterization of this filtration and as a corollary, we obtain an expression for the second Betti number of $\mathfrak{g}$. Using this expression, we find bounds for the second Betti number and derive a cohomological criterium for the existence of certain central extensions of $\mathfrak{g}$.
\end{abstract}
\maketitle
\section{Introduction}
Let $\mathfrak{g}$ be a finite dimensional $p$-step nilpotent Lie algebra and let $M$ be a finite dimensional $\mathfrak{g}$-module, both over an arbitrary field $k$.
$M$ is called a nilpotent $\mathfrak{g}$-module if the inductively defined increasing sequence of submodules $0=M_0 \subset M_1 \subset \ldots \subset M_{i-1} \subset M_i \subset \ldots $, with
\[ M_{i} := \{ m \in M \ | \ xm \in M_{i-1} \ \mbox{for all} \ x \in \mathfrak{g}\} \]
equals $M$ after a finite number of steps. $M$ is called $q$-step nilpotent if $M_{q-1}\neq M$ and $M_q=M$. \\
Assuming that $M$ is a $q$-step nilpotent $\mathfrak{g}$-module, there exists a well-known filtration of the cohomology space $\mathrm{H}^2(\mathfrak{g},M)$.
\begin{definition} \label{def: intro filt H^2}\rm For each $n \in \mathbb{N}$ and each $r\geq0$, one defines $F_r\mathrm{Hom}_k(\Lambda^n(\mathfrak{g}),M)$ to be
\[  \{ f \in \mathrm{Hom}_k(\Lambda^n(\mathfrak{g}),M) \ | \ f(x_1 \wedge \ldots \wedge x_n) \in M_{r-i_1-\ldots - i_n+1} \ \mbox{if} \  x_1 \wedge \ldots \wedge x_n \in \mathfrak{g}^{i_1}\wedge \ldots \wedge \mathfrak{g}^{i_n}\}. \]
Here, the spaces $\mathfrak{g}^{i}$ are the terms in the lower central series of $\mathfrak{g}$, i.e. $\mathfrak{g}^1=\mathfrak{g}$ and $\mathfrak{g}^{i+1}=[\mathfrak{g}^i,\mathfrak{g}]$.
This way, one obtains a filtration of the Chevalley-Eilenberg complex, that induces a filtration on the second cohomology.
\[ \ldots \subset F_{i-1}\mathrm{H}^2(\mathfrak{g},M) \subset F_i\mathrm{H}^2(\mathfrak{g},M) \subset \ldots \subset \mathrm{H}^2(\mathfrak{g},M). \]
\end{definition}
A case of particular interest is the one with $M=\mathfrak{g}$, where $\mathfrak{g}$ is given the adjoint $\mathfrak{g}$-module structure. Then $\mathfrak{g}$ is $p$-step nilpotent $\mathfrak{g}$-module, the sequence $0=M_0 \subset M_1 \subset \ldots \subset M_{i-1} \subset M_i \subset \ldots $ becomes the central ascending series of $\mathfrak{g}$ and we obtain a filtration $\ldots \subset F_{i-1}\mathrm{H}^2(\mathfrak{g},\mathfrak{g}) \subset F_i\mathrm{H}^2(\mathfrak{g},\mathfrak{g}) \subset \ldots \subset \mathrm{H}^2(\mathfrak{g},\mathfrak{g})$. It is well known that we can interpret the space $\mathrm{H}^2(\mathfrak{g},\mathfrak{g})$ as the set of infinitesimal deformation classes of $\mathfrak{g}$. Moreover, $\mathrm{H}^2(\mathfrak{g},\mathfrak{g})$ and the particular filtration described here play an important role in the study of the complex variety of nilpotent Lie algebras: let $\mathcal{L}^n$ be the variety of complex $n$-dimensional Lie algebras (see \cite{GozeKhakimdjanov}) and denote by $\mathcal{N}^n_i$ the Zariski closed algebraic subset of $\mathcal{L}^n$ consisting of nilpotent Lie algebras of class at most $i$, then the following holds (see \cite{GozeKhakimdjanov}).
\begin{theorem} Let $\mathfrak{g}$ be a $p$-step nilpotent Lie algebra of dimension $n$ over the complex numbers and let $q\geq p$ be an integer. Then the Zariski tangent space in $\mathfrak{g}$ to $\mathcal{N}^n_q$ consists of the $2$-cocycles $f \in \mathrm{Hom}_{\mathbb{c}}(\Lambda^2(\mathfrak{g}),\mathfrak{g})$ such that $\overline{f} \in  F_{q}\mathrm{H}^2(\mathfrak{g},\mathfrak{g})$.
\end{theorem}
If $M$ be a $q$-step nilpotent $\mathfrak{g}$-module, then every abelian extension of $\mathfrak{g}$ by $M$ is nilpotent of class $r$, with $\max\{p,q\}\leq r \leq p+q$. This observation suggest a filtration of the space of equivalence classes of abelian extensions of $\mathfrak{g}$ by $M$, which we denote by $\mathrm{Ext}(\mathfrak{g},M)$.
\begin{definition} \label{def: intro ext filt}\rm
 Let $\mathfrak{g}$ be a $p$-step nilpotent Lie algebra and let $M$ be a $q$-step nilpotent $\mathfrak{g}$-module. Take $r \in \mathbb{Z}$ such that $p \leq r \leq p+q$. With $F_r\mathrm{Ext}(\mathfrak{g},M)$ we mean the subset of $\mathrm{Ext}(\mathfrak{g},M)$ containing the equivalence classes of extensions
 \[ 0 \rightarrow M \rightarrow \mathfrak{e} \xrightarrow{\pi} \mathfrak{g} \rightarrow 0 \]
 such that $\mathfrak{e}$ is nilpotent of class at most $r$.
This defines a filtration
\[ \small F_{\max\{p,q\}}\mathrm{Ext}(\mathfrak{g},M) \subseteq  F_{\max\{p,q\}+1}\mathrm{Ext}(\mathfrak{g},M) \subseteq \ldots \subseteq F_{p+q-1}\mathrm{Ext}(\mathfrak{g},M) \subseteq F_{p+q}\mathrm{Ext}(\mathfrak{g},M)=\mathrm{Ext}(\mathfrak{g},M). \]
\end{definition}
It turns out that, under the bijective correspondence between $\mathrm{H}^2(\mathfrak{g},M)$ and $\mathrm{Ext}(\mathfrak{g},M)$, the spaces $F_r\mathrm{Ext}(\mathfrak{g},M)$ correspond to the spaces $F_r\mathrm{H}^2(\mathfrak{g},M)$ from Definition \ref{def: intro filt H^2} (see \cite{Vergne}). We would like to obtain yet another way of looking at the spaces $F_{r}\mathrm{H}^2(\mathfrak{g},M)$. To do this, we need the following definition.
\begin{definition} \rm Let $\mathfrak{g}$ be a $p$-step nilpotent Lie algebra with $|\frac{\mathfrak{g}}{[\mathfrak{g},\mathfrak{g}]}|=n$. A \emph{free $p$-step nilpotent extension} of $\mathfrak{g}$ is a short exact sequence of Lie algebras
\[ 0 \rightarrow \mathfrak{n} \rightarrow \mathfrak{f}_{n,p} \xrightarrow{\pi} \mathfrak{g} \rightarrow 0 \]
such that $\mathfrak{n} \subset [\mathfrak{f}_{n,p},\mathfrak{f}_{n,p}]$, where $\mathfrak{f}_{n,p}$ is the free $p$-step nilpotent Lie algebra on $n$ generators. We say $\mathfrak{g}$ is of \emph{depth} $d$ if $d$ is the largest integer such that $\mathfrak{n}\subseteq \mathfrak{f}_{n,p}^{d}$. \\
For integers $r \geq p$, a \emph{free $r$-step nilpotent extension} of $\mathfrak{g}$ is a short exact sequence of Lie algebras
\[ 0 \rightarrow \mathfrak{n'} \rightarrow \mathfrak{f}_{n,r} \xrightarrow{\pi'} \mathfrak{g} \rightarrow 0 \]
where $\mathfrak{f}_{n,r}$ is the free $r$-step nilpotent Lie algebra on $n$ generators and $\pi'$ is the composition of the canonical projection $\mathfrak{f}_{n,r} \rightarrow  \mathfrak{f}_{n,p}$ with a morphism $\pi: \mathfrak{f}_{n,p} \xrightarrow{\pi} \mathfrak{g}$, coming from a free $p$-step nilpotent extension of $\mathfrak{g}$.
\end{definition}
Every $p$-step nilpotent Lie algebra has a free $r$-step nilpotent extension, for any $r \geq p$, that is unique up to a certain notion of equivalence. \\ \\
We obtain the following characterization of the spaces $F_{r}\mathrm{H}^2(\mathfrak{g},M)$.
\begin{theorem} \label{th: intro} Let $\mathfrak{g}$ be a $p$-step nilpotent Lie algebra, $M$ a $q$-step nilpotent $\mathfrak{g}$-module and
\[ 0 \rightarrow \mathfrak{n} \rightarrow \mathfrak{f}_{n,r} \xrightarrow{\pi} \mathfrak{g}\rightarrow 0 \]
a free $r$-step nilpotent extension of $\mathfrak{g}$. Then
\[ F_r\mathrm{H}^2(\mathfrak{g},M) = \mathrm{Ker} \ \Big( \pi^2: \mathrm{H}^2(\mathfrak{g},M) \rightarrow \mathrm{H}^2(\mathfrak{f}_{n,r},M)\Big)\]
for every integer $r$ such that $\max\{p,q\} \leq r \leq p+q$.
Furthermore, for every integer $r$ such that $\max\{p,q\} \leq r \leq p+q$, we have an exact sequence
\[ 0 \rightarrow \mathrm{H}^1(\mathfrak{g},M) \rightarrow \mathrm{H}^1(\mathfrak{f}_{n,r},M) \rightarrow \mathrm{Hom}_{\mathfrak{g}}(\frac{\mathfrak{n}}{[\mathfrak{n},\mathfrak{n}]},M) \xrightarrow{d} F_r\mathrm{H}^2(\mathfrak{g},M) \rightarrow 0 \]
with $d(f)=\overline{f \circ \varphi_r}$,
 where $\varphi_r$ is a $2$-cocycle corresponding to
\[ 0 \rightarrow \frac{\mathfrak{n}}{[\mathfrak{n},\mathfrak{n}]} \rightarrow \frac{\mathfrak{f}_{n,r}}{[\mathfrak{n},\mathfrak{n}]} \xrightarrow{\pi} \mathfrak{g}\rightarrow 0. \]
If $M$ is a trivial $\mathfrak{g}$-module, then the map $d$ becomes an isomorphism.
\end{theorem}
As a first corollary, we obtain a formula for $|F_2\mathrm{H}^2(\mathfrak{g},k)|$ and $|F_2\mathrm{H}^2(\mathfrak{g},\mathfrak{g})|$ in the case where $\mathfrak{g}$ is a $2$-step nilpotent Lie algebra.
\begin{corollary} Let $\mathfrak{g}$ be a $2$-step nilpotent Lie algebra with $b_1(\mathfrak{g})=n$, then
\begin{eqnarray*}
|F_2\mathrm{H}^2(\mathfrak{g},k)| & = & \binom{n}{2}+n - |\mathfrak{g}|; \\
|F_2\mathrm{H}^2(\mathfrak{g},\mathfrak{g})| &=& \binom{n+1}{2}|Z(\mathfrak{g})|- n|\mathfrak{g}| - |\mathfrak{g}||Z(\mathfrak{g})|+|\mathrm{Der}(\mathfrak{g})|.
\end{eqnarray*}
Here, $Z(\mathfrak{g})$ is the center of $\mathfrak{g}$ and $\mathrm{Der}(\mathfrak{g})$ is the derivation algebra of $\mathfrak{g}$.
\end{corollary}
The Betti numbers of a finite dimensional Lie algebra over a field $k$ are by definition the dimensions of the different cohomology spaces with trivial coefficients $k$. For an $m$-dimensional nilpotent Lie algebra $\mathfrak{g}$, it is known that the Betti numbers satisfy $2 \leq b_i(\mathfrak{g})\leq \binom{m}{i}$ for all $i \in \{1,\ldots,m-1\}$ (see \cite{Dixmier}), with strict upper bound if and only if $\mathfrak{g}$ is non-abelian (see \cite{Nomizu}).
The Betti numbers of $\mathfrak{g}$ also satisfy a \emph{Golod-Shafarevich} type inequality: $\frac{b_1(\mathfrak{g})^2}{4} \leq b_2(\mathfrak{g})$ (see \cite{Koch}) and some \emph{Euler-Poincar\'{e}} type inequalities: $1 \leq \sum_{k=0}^i(-1)^{k+i}b_k(\mathfrak{g})$ for all $i \in \{0,\ldots,m-1\}$ (see \cite{Malliavin}). More recently, it was proven in \cite{Cairns} that if $\mathfrak{g}$ is non-abelian nilpotent and has dimension $m \geq 3$, one has $b_i(\mathfrak{g}) \leq \binom{m}{i}-\binom{m-2}{i-1}$. \\
Despite these results, the behavior of the Betti numbers of nilpotent Lie algebras is still unknown for the most part and several problems remain unsolved. The most famous perhaps, is the \emph{Toral rank conjecture} (attributed to S. Halperin, see \cite{Halperin}), which claims that for an $m$-dimensional nilpotent Lie algebra $\mathfrak{g}$, one has $\sum_{i=0}^mb_i(\mathfrak{g})\geq 2^z$, where $z$ denotes the dimension of $Z(\mathfrak{g})$, the center of $\mathfrak{g}$. This conjecture has been proven for $2$-step nilpotent Lie algebras (see \cite{Deninger}) and for nilpotent Lie algebras $\mathfrak{g}$ with $| Z(\mathfrak{g})|\leq 5$ or $| \frac{\mathfrak{g}}{Z(\mathfrak{g})} |\leq 7$ or $| \mathfrak{g} |\leq 14$ (see \cite{Cairns}). In fact, for $2$-step nilpotent Lie algebras, a  better lower bound was achieved in \cite{Tirao}. \\ \\
\indent Applying Theorem \ref{th: intro} to trivial $\mathfrak{g}$-modules, we obtain the following expression for the second Betti number.
\begin{corollary} \label{cor: intro betto} Let $\mathfrak{g}$ be a $p$-step nilpotent Lie algebra with free $p$-step nilpotent extension
\[ 0 \rightarrow \mathfrak{n} \rightarrow \mathfrak{f}_{n,p} \xrightarrow{\pi} \mathfrak{g} \rightarrow 0, \]
then
\begin{eqnarray*}
|F_p\mathrm{H}^2(\mathfrak{g},k)| &=& |\frac{\mathfrak{n}}{[\mathfrak{f}_{n,p},\mathfrak{n}]}|;\\
|\mathrm{H}^2(\mathfrak{g},k)|&=&  |\frac{\mathfrak{n}}{[\mathfrak{f}_{n,p},\mathfrak{n}]}| + b_2(\mathfrak{f}_{n,p})-|[i(\mathfrak{n}),\mathfrak{f}_{n,p+1}]\cap \mathfrak{f}_{n,p+1}^{p+1}|,
\end{eqnarray*}
where $i: \mathfrak{f}_{n,p} \rightarrow \mathfrak{f}_{n,p+1}$ is the canonical vector space inclusion.
\end{corollary}
We also obtain a necessary and sufficient cohomological condition for the existence of central extensions of class $p+1$, i.e. short exact sequences $0 \rightarrow M \rightarrow \mathfrak{e} \rightarrow \mathfrak{g} \rightarrow 0$ where $M$ is contained in the center of $\mathfrak{e}$ and $\mathfrak{e}$ is $(p+1)$-step nilpotent.
\begin{corollary}
Let $\mathfrak{g}$ be a $p$-step nilpotent Lie algebra with free $p$-step nilpotent extension
\[ 0 \rightarrow \mathfrak{n} \rightarrow \mathfrak{f}_{n,p} \xrightarrow{\pi} \mathfrak{g} \rightarrow 0. \] Then $\mathfrak{g}$ admits a central extension of class $p+1$ if and only if
$| \frac{\mathfrak{n}}{[\mathfrak{f}_{n,p},\mathfrak{n}]}| < b_2(\mathfrak{g})$.
In particular, if $\mathfrak{n}$ is generated by a single element $X \in [\mathfrak{f}_{n,p},\mathfrak{f}_{n,p}]$, then $\mathfrak{g}$ admits a central extension of class $p+1$.
\end{corollary}
The \emph{type} of $\mathfrak{g}$ is defined to be the sequence of numbers $(n_1,\ldots,n_i,\ldots,n_p) $, where $n_i$ equals the dimension of  $\mathfrak{g}^{i}/\mathfrak{g}^{i+1}$.  As an application of Corollary \ref{cor: intro betto}, we obtain bounds for the second Betti number of $\mathfrak{g}$ in terms of the type and depth of $\mathfrak{g}$.
\begin{corollary} Let $\mathfrak{g}$ be a $p$-step nilpotent Lie algebra of type $(n_1,n_2,\ldots,n_{p})$, dimension $m$, depth $d$ and with $p$-step nilpotent extension
\[ 0 \rightarrow \mathfrak{n} \rightarrow \mathfrak{f}_{n_1,p} \xrightarrow{\pi} \mathfrak{g} \rightarrow 0. \]
Define $c:= b_2(\mathfrak{f}_{n_1,d-1})-n_{d} $ and $C:= \binom{n_1}{2}+(m-n_1)(n_1-1)$ . In general, we have
\[ c \leq b_2(\mathfrak{g}) \leq C.\]
More specifically, if $\mathfrak{n}$ is generated by Lie monomials of degree $d$, we have
\[ c \leq b_2(\mathfrak{g}) \leq \min\{c + b_2(\mathfrak{f}_{n_1,p})-b_2(\mathfrak{f}_{n_1,p-1})+n_p,C \} \]
and if $\mathfrak{g}$ is of depth $p$ then
\[ c + \max{\{0,b_2(\mathfrak{f}_{n_1,p})-n_1b_2(\mathfrak{f}_{n_1,p-1})+n_1n_p\}} \leq b_2(\mathfrak{g}) \leq \min\{c + b_2(\mathfrak{f}_{n_1,p})-b_2(\mathfrak{f}_{n_1,p-1})+n_p,C \}. \]
\end{corollary}
\section{Preliminaries and notations} All vector spaces considered in this paper are finite dimensional over a fixed field $k$. If $V$ is a vector space, then we denote the dimension of $V$ by $|V|$. Now suppose $\mathfrak{g}$ is a Lie algebra over the field $k$. If $S$ is a subset of $\mathfrak{g}$, then the subalgebra/ideal generated by $S$ is by definition the smallest (via inclusion) subalgebra/ideal of $\mathfrak{g}$ containing $S$. It is not difficult to see that these spaces always exist and are unique. If we say that $\mathfrak{g}$ is generated by elements $X:=\{x_1,\ldots,x_{n}\}$ in $\mathfrak{g}$, we mean that the subalgebra generated by $X$ equals $\mathfrak{g}$. A generating set $X$ for $\mathfrak{g}$ is called minimal if no proper subset of $X$ generates $\mathfrak{g}$. \\
\indent The \emph{lower central series} of $\mathfrak{g}$ is the following decreasing sequence of ideals
\[ \mathfrak{g}=\mathfrak{g}^1  \supset [\mathfrak{g},\mathfrak{g}]=\mathfrak{g}^2 \supset  [\mathfrak{g},[\mathfrak{g},\mathfrak{g}]]=\mathfrak{g}^3 \supset \ldots \supset [\mathfrak{g},\mathfrak{g}^{i}]=\mathfrak{g}^{i+1} \supset \ldots \ .\]
For $p \in \mathbb{N}_0$, we say $\mathfrak{g}$ is \emph{$p$-step nilpotent} if $\mathfrak{g}^{p}\neq 0$ and $\mathfrak{g}^{p+1}=0$. If this is the case, then it is clear that $\mathfrak{g}^{p}$ is contained in the center $Z(\mathfrak{g})$ of $\mathfrak{g}$. If $\mathfrak{g}$ is $p$-step nilpotent,
the \emph{type} of $\mathfrak{g}$ is the sequence of numbers $(n_1,\ldots,n_i,\ldots,n_p) $, where $n_i = |\mathfrak{g}^{i}/\mathfrak{g}^{i+1}|$. It follows that $n_1 =|\frac{\mathfrak{g}}{[\mathfrak{g},\mathfrak{g}]}|$, $n_{p}=| \mathfrak{g}^{p}|$ and $\sum_{i=1}^p n_i=|\mathfrak{g}|$. \\
\indent Let $\mathfrak{f}_{n}$ be the free Lie algebra over $k$ on generators $X=\{x_1,\ldots,x_{n}\}$. The free Lie algebra is graded. Indeed, we can write
$\mathfrak{f}_{n}= H_1(n) \oplus H_2(n) \oplus \ldots \oplus H_{p}(n) \oplus \ldots $, where $H_i(n)$ is the vector space spanned by the Lie monomials in $X$ of degree $i$, such that $[H_i(n),H_j(n)]\subset H_{i+j}(n)$. Explicit bases can be constructed for $H_i(n)$, for example the Hall basis (see \cite{Serre}). Moreover, it is known that
$|H_i(n)|= \frac{1}{i} \sum_{d|i}\mu(d)n^{\frac{i}{d}}$ where $\mu$ equals the M\"{o}bius function, which is defined as follows
\[ \mu(d)= \left\{\begin{array}{ccc}
             1 & \mbox{if} & \mbox{$d \in \mathbb{\mathbb{Z}}^{+}$, square-free, with an even number of prime factors} \\
             -1 & \mbox{if} & \mbox{$d \in \mathbb{\mathbb{Z}}^{+}$, square-free, with an odd number of prime factors} \\
             0 &  & \mbox{otherwise}.
           \end{array}\right.\]
The free $p$-step nilpotent Lie algebra $\mathfrak{f}_{n,p}$ on generators $X=\{x_1,\ldots,x_{n}\}$ is defined as $\mathfrak{f}_{n}/\mathfrak{f}_{n}^{p+1}$. The universal property of the free Lie algebra implies that any set map from $X$ to a $p$-step nilpotent Lie algebra can be uniquely extended to a Lie algebra homomorphism from $\mathfrak{f}_{n,p}$ to this given $p$-step nilpotent Lie algebra. The free $p$-step nilpotent Lie algebra is graded, as it inherits its grading from the free Lie algebra. By abusing notation a little, we can write
$\mathfrak{f}_{n,p}= H_1(n) \oplus H_2(n) \oplus \ldots \oplus H_{p}(n)$, with $[H_i(n),H_j(n)]=0$ if $i+j>p$. \\
An introduction to (nilpotent) Lie algebras can, for example, be found in \cite{Corwin}, \cite{Jacobson}, \cite{Serre} and \cite{GozeKhakimdjanov}. \\
\indent Let $M$ be a $\mathfrak{g}$-module. One can compute $\mathrm{H}^{\ast}(\mathfrak{g},M)$, i.e. the cohomology of $\mathfrak{g}$ with coefficients in $M$, using the Chevalley-Eilenberg complex of $\mathfrak{g}$
\[ 0 \rightarrow M \xrightarrow{d^0} \mathrm{Hom}_k(\mathfrak{g},M) \xrightarrow{d^1}  \mathrm{Hom}_k(\Lambda^{2}(\mathfrak{g}),M)\xrightarrow{d^2} \ldots \rightarrow \mathrm{Hom}_k(\Lambda^{i}(\mathfrak{g}),M) \xrightarrow{d^i}\ldots \]
where $\Lambda^{i}(\mathfrak{g})$ denotes the $i$-th exterior product of $\mathfrak{g}$. Here, $d^0(m)(x)$ equals $xm$ and for $i \geq 1$, the coboundary $d^i(\omega)$ of a $i$-cochain is the $(i+1)$-cochain
\begin{eqnarray*}
d^i(\omega)(x_1 \wedge \ldots \wedge x_{i+1}) &=&   \sum_{k=1}^{i+1}(-1)^{k+1}x_k\omega(x_1 \wedge \ldots \wedge \hat{x_k} \wedge \ldots \wedge x_{i+1})  + \\
& & \sum_{k < l}(-1)^{k+l} \omega( [x_k,x_l] \wedge x_1 \wedge \ldots \wedge \hat{x_k} \wedge \ldots \wedge \hat{x_l} \wedge \ldots \wedge x_{i+1}).
\end{eqnarray*}
By definition, we have $\mathrm{H}^{i}(\mathfrak{g},M)=\frac{\mathrm{Ker}\ d^i}{\mathrm{Im} \ d^{i-1}}$. An element $\omega \in  \mathrm{Hom}_k(\Lambda^{i}(\mathfrak{g}),M)$ is called an $i$-cocycle if $d^i(\omega)=0$, it is called a $i$-coboundary if $\omega=d^{i-1}(\mu)$ for some $\mu \in  \mathrm{Hom}_k(\Lambda^{i-1}(\mathfrak{g}),M)$.
If $\omega \in \mathrm{Hom}_k(\Lambda^{2}(\mathfrak{g}),k)$ is a $2$-cocycle of $\mathfrak{g}$, then we will write $\overline{\omega}$ for the corresponding class in $\mathrm{H}^2(\mathfrak{g},M)$. More generally, the cohomology of $\mathfrak{g}$ can also be computed using derived functors. Indeed, we have $\mathrm{H}^{\ast}(\mathfrak{g},M)=\mathrm{Ext}^{\ast}_{U(\mathfrak{g})}(k,M)$, where $U(\mathfrak{g})$ is the universal enveloping algebra of $\mathfrak{g}$ and $\mathrm{Ext}^n_{U(\mathfrak{g})}(k,M)$ is the $n$-th right derived functor of $\mathrm{Hom}_{U(\mathfrak{g})}(k,-)$, evaluated at $M$. For details on homological algebra and (co)homology of Lie algebras, we refer the reader to \cite{Knapp} and \cite{Weibel}. \\
\indent For $i \in \{1,\ldots,m=|\mathfrak{g}|\}$, the \emph{Betti number} $b_i(\mathfrak{g})$ of $\mathfrak{g}$ is by definition the dimension of $\mathrm{H}^i(\mathfrak{g},k)$. Here $k$ is viewed as a trivial $\mathfrak{g}$-module. One can verify that $b_0(\mathfrak{g})=1$, $b_1(\mathfrak{g})=|\frac{\mathfrak{g}}{[\mathfrak{g},\mathfrak{g}]}|$ and $0 \leq b_i(\mathfrak{g}) \leq \binom{n}{i}$ for all $i \in \{0,1,\ldots,m\}$. It is also known that, if $\mathfrak{g}$ is a nilpotent Lie algebra, \emph{Poincar\'{e} duality} holds: $b_i(\mathfrak{g})=b_{m-i}(\mathfrak{g})$ for all $i \in \{0,1,\ldots,n\}$. Several other inequalities for the Betti numbers of nilpotent Lie algebra are provided in the introduction. Finally, we remark that $b_2(\mathfrak{f}_{n,p})=|H_{p+1}(n)|$, which can be computed using the M\"{o}bius function. \\
\indent Given a short exact sequence of Lie algebras
\begin{equation} \label{eq: general short exact sequence}
0 \rightarrow \mathfrak{n} \rightarrow \mathfrak{h} \xrightarrow{\pi} \mathfrak{g} \rightarrow 0 \end{equation}
and an $\mathfrak{h}$-module $M$, one can consider the associated \emph{Hochschild-Serre spectral sequence} (see \cite{HochSerre},\cite{McClearly} and \cite{Weibel})
\[ \mathrm{E}_2^{p,q}= \mathrm{H}^p(\mathfrak{h},\mathrm{H}^q(\mathfrak{n},M)) \Rightarrow \mathrm{H}^{p+q}(\mathfrak{g},M). \]
This spectral sequence gives rise to an exact sequence of low degree terms
\[ 0 \rightarrow \mathrm{H}^1(\mathfrak{g},M^{\mathfrak{n}}) \rightarrow \mathrm{H}^1(\mathfrak{e},M) \rightarrow \mathrm{H}^1(\mathfrak{n},M)^{\mathfrak{g}} \xrightarrow{d} \mathrm{H}^2(\mathfrak{g},M^{\mathfrak{n}}) \rightarrow  \mathrm{H}^2(\mathfrak{e},M) \]
for every $\mathfrak{h}$-module $M$. Here, $d$ is the $d_2^{0,1}$-differential from the spectral sequence. The other maps are the well-known restriction and inflation maps. Note that an elementary working knowledge of spectral sequences will be assumed throughout this text. \\
\indent
Let $\mathfrak{g}$ be a Lie algebra and $M$ a $\mathfrak{g}$-module. An extension of $\mathfrak{g}$ by $M$ is a short exact sequence $0 \rightarrow M \rightarrow \mathfrak{e} \xrightarrow{\pi} \mathfrak{g} \rightarrow 0$ of Lie algebras, where $M$ is viewed as an abelian Lie algebra such that the $\mathfrak{g}$-module structure induced by the extension coincides with the original $\mathfrak{g}$-module structure on $M$ (i.e. for all $m \in M$: $[y,m]=xm$ for all $y \in \mathfrak{e}$ such that $\pi(y)=x$). We say such an extension is \emph{split}, if there exists a Lie algebra homomorphism $\pi: \mathfrak{g} \rightarrow \mathfrak{e}$ such that $\pi \circ s=\mathrm{Id}$. Two extensions of $\mathfrak{g}$ by $M$, $0 \rightarrow M \rightarrow \mathfrak{e}_i \rightarrow \mathfrak{g} \rightarrow 0$ for $i \in \{1,2\}$, are called equivalent if there exists a Lie algebra homomorphism $\theta: \mathfrak{e}_1 \rightarrow \mathfrak{e}_2$ such that
 \[ \xymatrix{  0 \ar[r] & M \ar[r] \ar[d]^{\mathrm{Id}} & \mathfrak{e}_1 \ar[r] \ar[d]^{\theta} & \mathfrak{g} \ar[d]^{\mathrm{Id}} \ar[r] & 0  \\
  0 \ar[r] & M \ar[r]^{i}  & \mathfrak{e}_2 \ar[r]^{\pi}  & \mathfrak{g} \ar[r] & 0} \]
commutes. This induces an equivalence relation on the set of all extensions of $\mathfrak{g}$ by $M$. One can easily check that all split extensions of $\mathfrak{g}$ by $M$ are equivalent. Denote the set of equivalence classes by $\mathrm{Ext}(\mathfrak{g},M)$. The following theorem is well-known.
\begin{theorem} There is a bijective correspondence between  $\mathrm{H}^2(\mathfrak{g},M)$ and $\mathrm{Ext}(\mathfrak{g},M)$, such that the class of split extensions corresponds to the zero element in $\mathrm{H}^2(\mathfrak{g},M)$.
\end{theorem}
We will now describe how this correspondence can be obtained. Given a $2$-cocycle $\omega \in \mathrm{Hom}_k(\Lambda^2(\mathfrak{g}),M)$, define the Lie algebra $\mathfrak{h}:=M\oplus \mathfrak{g}$ with Lie bracket given by $[(m,x),(n,y)]:=(xn-ym + \omega(x \wedge y),[x,y])$ (the fact the $\omega$ is a $2$-cocycle ensures that this is in fact a Lie bracket). Then $0 \rightarrow M \xrightarrow{i} \mathfrak{h} \xrightarrow{p} \mathfrak{g} \rightarrow 0$ is an extension of $\mathfrak{g}$ by $M$, where $i$ and $p$ are the obvious inclusion and projection maps. Write $\alpha$ for the equivalence class represented by this extension. Then one can check that the map $\Theta: \mathrm{H}^2(\mathfrak{g},M) \rightarrow \mathrm{Ext}(\mathfrak{g},M): \overline{\omega} \mapsto \alpha$ is well defined. \\
Now let  $0 \rightarrow M \rightarrow \mathfrak{e} \xrightarrow{\pi} \mathfrak{g} \rightarrow 0$ be an extension of $\mathfrak{g}$ by $M$ and consider its associated exact sequence of low degree terms, with coefficients in $M$ viewed as an $\mathfrak{e}$-module via $\pi$
\[ 0 \rightarrow \mathrm{H}^1(\mathfrak{g},M) \rightarrow \mathrm{H}^1(\mathfrak{e},M) \rightarrow \mathrm{H}^1(M,M)^{\mathfrak{g}} \xrightarrow{d} \mathrm{H}^2(\mathfrak{g},M) \rightarrow  \mathrm{H}^2(\mathfrak{e},M). \]
By duality, we know that $\mathrm{H}^1(M,M)^{\mathfrak{g}}\cong\mathrm{Hom}_{\mathfrak{g}}(M,M)$. Denote the element in $\mathrm{H}^1(M,M)^{\mathfrak{g}}$ that corresponds to the identity map in $\mathrm{Hom}_{\mathfrak{g}}(M,M)$, under this isomorphism, by $\mathrm{Id}_M$. One can check that the map $\Psi: \mathrm{Ext}(\mathfrak{g},M) \rightarrow \mathrm{H}^2(\mathfrak{g},M)$ that maps the equivalence class of the extension $0 \rightarrow M \rightarrow \mathfrak{e} \xrightarrow{\pi} \mathfrak{g} \rightarrow 0$ to $d(\mathrm{Id}_M)$ is well-defined. It turns out that the maps $\Psi$ and $\Theta$ are mutual inverses of each other, so we obtain the desired bijection.
For details we refer the reader to \cite{Weibel}. \\
\indent
Finally, let us describe an important construction that will be used in the proof of our main theorem.
Consider a short exact sequence of Lie algebras $(\ref{eq: general short exact sequence})$ and a Lie algebra homomorphism $\mathfrak{p} \xrightarrow{p} \mathfrak{g}$.
Now define $\mathfrak{e}$ as the following subalgebra of the product Lie algebra $\mathfrak{h} \oplus \mathfrak{p}$
\[ \mathfrak{e}:=\{ (x,y) \in \mathfrak{h} \oplus \mathfrak{p} \ | \ \pi(x)=p(y) \}.\]
One easily sees that $(n,0) \in \mathfrak{e}$ for all $n \in \mathfrak{n}$, so we have an inclusion $\mathfrak{n} \rightarrow \mathfrak{e}$.
Now let $\psi_1$ be the restriction to $\mathfrak{e}$ of the projection $\mathfrak{h} \oplus \mathfrak{p} \rightarrow \mathfrak{h}$ and let $\psi_2$ be the restriction to $\mathfrak{e}$ of the projection $\mathfrak{h} \oplus \mathfrak{p} \rightarrow \mathfrak{p}$. The reader can check that
$\mathfrak{e}$ fits into a short exact sequence $0 \rightarrow \mathfrak{n} \rightarrow \mathfrak{e} \xrightarrow{\psi_2} \mathfrak{p} \rightarrow 0$, and a commutative diagram
\[ \xymatrix{  0 \ar[r] & \mathfrak{n} \ar[r] \ar[d]^{\mathrm{Id}} & \mathfrak{e} \ar[r]^{\psi_2} \ar[d]^{\psi_1} & \mathfrak{p} \ar[d]^{p} \ar[r] & 0  \\
  0 \ar[r] & \mathfrak{n} \ar[r]  & \mathfrak{h} \ar[r]^{\pi}  & \mathfrak{g} \ar[r] & 0.} \]
This construction will be referred to as the \emph{the pullback construction}. Indeed, one can verify that in the category of Lie algebras $(\mathfrak{e},\psi_1,\psi_2)$ is the pullback of $(\mathfrak{h} \xrightarrow{\pi} \mathfrak{g},\mathfrak{p} \xrightarrow{p} \mathfrak{g})$. Note that it follows from the construction that, if $\mathfrak{h}$ and $\mathfrak{p}$ are nilpotent of class at most $r$, $\mathfrak{e}$ is also nilpotent of class at most $r$.
\section{A filtration of the second cohomology space}
Throughout this section, let $\mathfrak{g}$ be a finite dimensional $p$-step nilpotent Lie algebra and $M$ a finite dimensional $\mathfrak{g}$-module. \\
\indent It is clear that, given any abelian extension of $\mathfrak{g}$ by $M$
\[ 0 \rightarrow M \rightarrow \mathfrak{e} \xrightarrow{\pi} \mathfrak{g} \rightarrow 0, \]
the Lie algebra $\mathfrak{e}$ is not necessarily nilpotent. However, if we consider extensions of $\mathfrak{g}$ with a particular type of $\mathfrak{g}$-module, then these extensions will yield nilpotent Lie algebras. Let us define this particular type of $\mathfrak{g}$-module.
\begin{definition} \rm A $\mathfrak{g}$-module $M$ is called $q$-step nilpotent if $q \geq 1$ is the smallest integer for which $M$ admits a strictly decreasing sequence of submodules
\[ 0=M^q \subsetneq M^{q-1} \subsetneq \ldots \subsetneq M^{j+1} \subsetneq M^{j} \subsetneq \ldots \subsetneq M^1 \subsetneq M^0 = M, \]
such that $\mathfrak{g}^i  M^j \subset M^{i+j}$, for all $i,j$ (we set $M^r=0$ for all $r\geq q$).
\end{definition}
The key property, for our purposes, of nilpotent $\mathfrak{g}$-modules is described in the following proposition.
\begin{proposition}\label{prop: bound nilp class} Let $\mathfrak{g}$ be a $p$-step nilpotent Lie algebra and $M$ a $\mathfrak{g}$-module. Consider an abelian extension of $\mathfrak{g}$ by $M$
\[ 0 \rightarrow M \rightarrow \mathfrak{e} \xrightarrow{\pi} \mathfrak{g} \rightarrow 0. \]
If $M$ is a $q$-step nilpotent $\mathfrak{g}$-module then $\mathfrak{e}$ is an $r$-step nilpotent Lie algebra with $p \leq r \leq p+q$. Moreover, if the extension is split, then $\mathfrak{e}$ is $r$-step nilpotent, with $r=\max\{p,q\}$. Conversely, if $\mathfrak{e}$ is an $r$-step nilpotent Lie algebra, then $M$ is a $q$-step nilpotent $\mathfrak{g}$-module with $r-p \leq q \leq r$.
\end{proposition}
\begin{proof}
The proof is a relatively easy exercise, left to the reader. \end{proof}
This proposition suggests a filtration of $\mathrm{Ext}(\mathfrak{g},M)$.
\begin{definition} \label{def: ext filt}\rm
 Let $\mathfrak{g}$ be a $p$-step nilpotent Lie algebra and $M$ a $q$-step nilpotent $\mathfrak{g}$-module. Take $r \in \mathbb{Z}$ such that $p \leq r \leq p+q$. With $F_r\mathrm{Ext}(\mathfrak{g},M)$ we mean the subset of $\mathrm{Ext}(\mathfrak{g},M)$ containing the equivalence classes of extensions
 \[ 0 \rightarrow M \rightarrow \mathfrak{e} \xrightarrow{\pi} \mathfrak{g} \rightarrow 0 \]
 such that $\mathfrak{e}$ is nilpotent of class at most $r$.
This defines a filtration
\[  F_{\max\{p,q\}}\mathrm{Ext}(\mathfrak{g},M) \subseteq  F_{\max\{p,q\}+1}\mathrm{Ext}(\mathfrak{g},M) \subseteq \ldots \subseteq F_{p+q-1}\mathrm{Ext}(\mathfrak{g},M) \subseteq F_{p+q}\mathrm{Ext}(\mathfrak{g},M)=\mathrm{Ext}(\mathfrak{g},M). \]
\end{definition}
If $M$ is a $q$-step nilpotent $\mathfrak{g}$-module, there exists a well-known filtration of $\mathrm{H}^2(\mathfrak{g},M)$ (see \cite{GozeKhakimdjanov}, \cite{Vergne} but mind the different indexing). To construct this filtration, we need a special filtration of $M$: $M_j:=0$ for all integer $j\leq 0$ and
\[ M_{i} := \{ m \in M \ | \ xm \in M_{i-1} \ \mbox{for all} \ x \in \mathfrak{g}\}, \]
for all integers $i\geq 1$.
One checks that $M_1=M^{\mathfrak{g}}:=\{ m \in M \ | \ xm=0, \ \forall x \in \mathfrak{g}\}$ and $M_{q}=M$. Also, if we set $M^i:=M_{q-i}$ for all $i \in \{0,\ldots,q\}$, then
\[ 0=M^q \subsetneq M^{q-1} \subsetneq \ldots \subsetneq M^{j+1} \subsetneq M^{j} \subsetneq \ldots \subsetneq M^1 \subsetneq M^0 = M \]
is a strictly decreasing sequence of $\mathfrak{g}$-submodules such that $\mathfrak{g}^i  M^j \subset M^{i+j}$ for all $i,j$.
\begin{definition} \label{def: filt H^2}\rm For each $n,r \in \mathbb{N}$, we define $F_r\mathrm{Hom}_k(\Lambda^n(\mathfrak{g}),M)$ to be
\[  \{ f \in \mathrm{Hom}_k(\Lambda^n(\mathfrak{g}),M) \ | \ f(x_1 \wedge \ldots \wedge x_n) \in M_{r-i_1-\ldots - i_n+1} \ \mbox{if} \  x_1 \wedge \ldots \wedge x_n \in \mathfrak{g}^{i_1}\wedge \ldots \wedge \mathfrak{g}^{i_n}\}. \]
One can verify that $F_{i-1}\mathrm{Hom}_k(\Lambda^n(\mathfrak{g}),M) \subset F_i\mathrm{Hom}_k(\Lambda^n(\mathfrak{g}),M)$ for all $i$ and $n$. Also, the differentials of the Chevalley-Eilenberg complex behave well in respect to this filtration, in the sense that $F_i\mathrm{Hom}_k(\Lambda^{\ast}(\mathfrak{g}),M)$ forms a sub-complex of $\mathrm{Hom}_k(\Lambda^{\ast}(\mathfrak{g}),M)$ for each $i$. Hence, by considering the inclusion maps on the cohomology level, we obtain a filtration on cohomology. In particular, we obtain a filtration
\[ \ldots \subset F_{i-1}\mathrm{H}^2(\mathfrak{g},M) \subset F_i\mathrm{H}^2(\mathfrak{g},M) \subset \ldots \subset \mathrm{H}^2(\mathfrak{g},M). \]
\end{definition}
The next result, due to Vergne (see, \cite{Vergne}) shows that, under the bijective correspondence between $\mathrm{H}^2(\mathfrak{g},M)$ and $\mathrm{Ext}(\mathfrak{g},M)$, the spaces $F_r\mathrm{Ext}(\mathfrak{g},M)$ from Definition \ref{def: ext filt} correspond to the spaces $F_r\mathrm{H}^2(\mathfrak{g},M)$ from Definition \ref{def: filt H^2}.
\begin{proposition} \label{prop: vergne}(Vergne, \cite{Vergne}) Let $\mathfrak{g}$ be a $p$-step nilpotent Lie algebra, $M$ a $q$-step nilpotent $\mathfrak{g}$-module and
consider the extension of $\mathfrak{g}$ by $M$
\[ 0 \rightarrow M \rightarrow \mathfrak{e} \rightarrow \mathfrak{g} \rightarrow 0. \]
Let $\alpha \in \mathrm{H}^2(\mathfrak{g},M)$ be the element  corresponding to the equivalence class of this extension. Take an integer $r$ such that $\max\{p,q\} \leq r \leq p+q$. Then $\mathfrak{e}$ is nilpotent of class at most $r$ if and only if $\alpha \in F_{r}\mathrm{H}^2(\mathfrak{g},M)$.
\end{proposition}
\begin{proof} Suppose $\alpha \in F_{r}\mathrm{H}^2(\mathfrak{g},M)$. It follows that $\alpha=\overline{f}$, where $f \in F_r\mathrm{Hom}_k(\Lambda^2(\mathfrak{g}),M)$. We may assume that $\mathfrak{e}=M \oplus \mathfrak{g}$ with Lie bracket $[(m,x),(n,y)]:=(xn-ym+f(x\wedge y),[x,y])$, and need to prove that $\mathfrak{e}^{r+1}=0$. Using induction and  the fact that $f \in F_r\mathrm{Hom}_k(\Lambda^2(\mathfrak{g}),M)$, one shows that $\mathfrak{e}^i \subset (M_{q-i+1}+M_{r-i+1},\mathfrak{g}^i)$ for all $i\geq 1$. Hence $\mathfrak{e}^{r+1} \subset (M_{q-r}+M_{0},\mathfrak{g}^{r+1})$. Since $\max\{p,q\} \leq r$, we have $M_{q-r}=0$ and $\mathfrak{g}^{r+1}=0$. We conclude that $\mathfrak{e}^{r+1}=0$. \\
Conversely, suppose that $\mathfrak{e}$ is nilpotent of class at most $r$. We will construct a cocycle $f \in F_r\mathrm{Hom}_k(\Lambda^2(\mathfrak{g}),M)$ such that $\overline{f}=\alpha$. Choose a linear map $\mathfrak{s} \rightarrow \mathfrak{e}$ such that $\pi \circ s=\mathrm{Id}$ and $s(\mathfrak{g}^i)\subset \mathfrak{e}^i$ for all $i$. It is well-known that $f$, with $f(x \wedge y):=[s(x),s(y)]-s([x,y])$ for all $x,y \in \mathfrak{g}$, is a cocycle such that $\overline{f}=\alpha$. It remains to show  that $f \in F_r\mathrm{Hom}_k(\Lambda^2(\mathfrak{g}),M)$. Take $x \in \mathfrak{g}^i$ and $y \in \mathfrak{g}^j$ for some $i,j$. Using the fact that $s(\mathfrak{g}^i)\subset \mathfrak{e}^i$ for all $i$, one easily checks that $f(x \wedge y) \subset \mathfrak{e}^{i+j}$. If $i+j\geq r+1$, then it follows from $\mathfrak{e}^{r+1}=0$ that  $f(x \wedge y) \in  M_{r-i-j+1}$. Suppose that $i+j \leq r$. Because $\mathfrak{e}^{r+1}=0$, it follows that $\mathfrak{e}^r \cap M \subset M^{\mathfrak{g}}=M_1$. Because $M_2$ consists exactly of those $m \in M$ such that $xm \in M_1$ for all $x \in \mathfrak{g}$, it follows that $\mathfrak{e}^{r-1} \cap M \subset M_2$. Proceeding in this manner, we find that $\mathfrak{e}^{i+j} \cap M \subset M_{r-i-j+1}$. Hence $f(x \wedge y) \in  M_{r-i-j+1}$, which shows that $f \in F_r\mathrm{Hom}_k(\Lambda^2(\mathfrak{g}),M)$.
\end{proof}
To summarize, we have a filtration
\[ F_{\max\{p,q\}}\mathrm{H}^2(\mathfrak{g},M) \subseteq F_{\max\{p,q\}+1}\mathrm{H}^2(\mathfrak{g},M) \subseteq \ldots \subseteq F_{p+q-1}\mathrm{H}^2(\mathfrak{g},M) \subseteq F_{p+q}\mathrm{H}^2(\mathfrak{g},M)=\mathrm{H}^2(\mathfrak{g},M) \]
defined in terms of conditions on the cocycle level, such that the elements in $F_r\mathrm{Ext}(\mathfrak{g},M)$ correspond to equivalence classes of extensions of $\mathfrak{g}$ by $M$, that are nilpotent of class at most $r$. \\ \\
\indent We would like to give a different characterization of the spaces $F_r\mathrm{H}^2(\mathfrak{g},M)$.
To do this, we first need the notion of a \emph{free nilpotent extension}.
\begin{definition} \rm Let $\mathfrak{g}$ be a $p$-step nilpotent Lie algebra with $b_1(\mathfrak{g})=n$. A \emph{free $p$-step nilpotent extension} of $\mathfrak{g}$ is a short exact sequence of Lie algebras
\[ 0 \rightarrow \mathfrak{n} \rightarrow \mathfrak{f}_{n,p} \xrightarrow{\pi} \mathfrak{g} \rightarrow 0 \]
such that $\mathfrak{n} \subset [\mathfrak{f}_{n,p},\mathfrak{f}_{n,p}]$, where $\mathfrak{f}_{n,p}$ is the free $p$-step nilpotent Lie algebra on $n$ generators. We say $\mathfrak{g}$ is of \emph{depth} $d$ if $d$ is the largest integer such that $\mathfrak{n}\subseteq \mathfrak{f}_{n,p}^{d}$. It follows that the induced map $\frac{\mathfrak{f}_{n,p} }{[\mathfrak{f}_{n,p},\mathfrak{f}_{n,p}]} \rightarrow \frac{\mathfrak{g}}{[\mathfrak{g},\mathfrak{g}]}$ is an isomorphism and that $d \in \{2,\ldots,p\}$. \\
For $r \geq p$, A \emph{free $r$-step nilpotent extension} of $\mathfrak{g}$ is a short exact sequence of Lie algebras
\[ 0 \rightarrow \mathfrak{n'} \rightarrow \mathfrak{f}_{n,r} \xrightarrow{\pi'} \mathfrak{g} \rightarrow 0 \]
where $\mathfrak{f}_{n,r}$ is the free $r$-step nilpotent Lie algebra on $n$ generators and $\pi'$ is the composition of the canonical projection $\mathfrak{f}_{n,r} \rightarrow  \mathfrak{f}_{n,p}$ with a morphism $\pi: \mathfrak{f}_{n,p} \xrightarrow{\pi} \mathfrak{g}$, coming from a free $p$-step nilpotent extension of $\mathfrak{g}$.
\end{definition}
\indent Note that every $p$-step nilpotent Lie algebra $\mathfrak{g}$ has a free $p$-step nilpotent extension. Indeed, choose elements $\{x_1,\ldots,x_{n}\}\subset \mathfrak{g}$ such that under the projection of $\mathfrak{g}$ onto $\frac{\mathfrak{g}}{[\mathfrak{g},\mathfrak{g}]}$, the elements  $\{x_1,\ldots,x_{n}\}$ are mapped to a basis $\{\overline{x_1},\ldots,\overline{x_{n}}\}$ for $\frac{\mathfrak{g}}{[\mathfrak{g},\mathfrak{g}]}$ (so $b_1(\mathfrak{g})=n$). This implies that $\{x_1,\ldots,x_{n}\}$ is a minimal generating set for $\mathfrak{g}$.
Now let $\mathfrak{f}_{n,p}$ be the free $p$-step nilpotent Lie algebra on generators $\{y_1,\ldots,y_{n}\}$. By the universal property of free $p$-step nilpotent Lie algebras, there exists a unique homomorphism $\pi$ of $\mathfrak{f}_{n,p}$ to $\mathfrak{g}$ such that $\pi(y_i)=x_i$ for all $i \in \{1,\ldots,n\}$. Because $\{x_1,\ldots,x_{n}\}$ is a generating set for $\mathfrak{g}$, it follows that $\pi$ is surjective.  Hence, we obtain a short exact sequence
$0 \rightarrow \mathfrak{n} \rightarrow \mathfrak{f}_{n,p} \xrightarrow{\pi} \mathfrak{g} \rightarrow 0$.
An element $Y$ of $\mathfrak{f}_{n,p}$ is always of the form $\sum_{i=1}^{n}\lambda_i y_i+ \sum_{j=1}^r[a_i,b_j]$, with $a_i,b_j \in \mathfrak{f}_{n,p}$. Under the composition of $\pi$ with the projection of $\mathfrak{g}$ onto $\frac{\mathfrak{g}}{[\mathfrak{g},\mathfrak{g}]}$, $Y$ is mapped to $\sum_{i=1}^{n}\lambda_i \overline{x_i}$. Since  $\{\overline{x_1},\ldots,\overline{x_{n}}\}$ is a basis of $\frac{\mathfrak{g}}{[\mathfrak{g},\mathfrak{g}]}$, $\pi(Y)=0$ implies that $\lambda_i=0$ for all $i \in \{1,\ldots,n\}$. This shows that $\mathfrak{n} \subset [\mathfrak{f}_{n,p},\mathfrak{f}_{n,p}]$. Hence, $0 \rightarrow \mathfrak{n} \rightarrow \mathfrak{f}_{n,p} \xrightarrow{\pi} \mathfrak{g} \rightarrow 0$ is a free $p$-step nilpotent extension of  $\mathfrak{g}$. It now easily follows that every $p$-step nilpotent Lie algebra also has a free $r$-step nilpotent extension, for every $r \geq p$. \\ \\

\indent It is clear that free $r$-step nilpotent extensions of a nilpotent Lie algebra $\mathfrak{g}$ are not unique, since they involve a particular choice of generators of $\mathfrak{g}$. However, they are equivalent in some sense.
\begin{proposition}\label{prop: unique nilp ext} Let $\mathfrak{g}$ be a $p$-step nilpotent Lie algebra. Two free $r$-step nilpotent extensions $0 \rightarrow \mathfrak{n}_i \rightarrow \mathfrak{f}_{n,q} \xrightarrow{\pi_i} \mathfrak{g} \rightarrow 0$ for $i=1,2$, of $\mathfrak{g}$ are always equivalent, meaning that there exists a $\theta \in \mathrm{Aut}(\mathfrak{f}_{n,r})$ such that $\pi_1 = \pi_2 \circ \theta$.
\end{proposition}
\begin{proof} Fix a set of generators $\{x_1,\ldots,x_{n}\}$ for $\mathfrak{f}_{n,r}$. Now define $a_i:=\pi_1(x_i)$ and $b_i:=\pi_2(x_i)$ for all $i \in \{1,\ldots,n\}$. Since $\{\overline{a_1},\ldots,\overline{a_{n}}\}$ and $\{\overline{b_1},\ldots,\overline{b_{n}}\}$ are two sets of bases for $\frac{\mathfrak{g}}{[\mathfrak{g},\mathfrak{g}]}$, we can find a linear invertible operator $T$ on $\frac{\mathfrak{g}}{[\mathfrak{g},\mathfrak{g}]}$ such that $T(a_i)=b_i$ for all $i$. Since $<x_1,\ldots,x_n>=\frac{\mathfrak{f}_{n,r}}{[\mathfrak{f}_{n,r},\mathfrak{f}_{n,r}]}$ is mapped isomorphically onto $\frac{\mathfrak{g}}{[\mathfrak{g},\mathfrak{g}]}$ by $\overline{\pi_1}$ and $\overline{\pi_2}$, we obtain an linear invertible map $S:=\overline{\pi_2}^{-1}\circ T \circ \overline{\pi_1}: <x_1,\ldots,x_n> \rightarrow <x_1,\ldots,x_n> $. Using the universal property of $\mathfrak{f}_{n,r}$, we can uniquely extend $S$ to a Lie algebra homomorphism $\theta: \mathfrak{f}_{n,r} \rightarrow \mathfrak{f}_{n,r}$. We leave it to the reader to check that $\theta$ has the desired properties, using the universal property of $\mathfrak{f}_{n,r}$.
\end{proof}
We can now state and prove the following characterization of the filtration of $\mathrm{H}^2(\mathfrak{g},M)$.
\begin{theorem} \label{th: filtration} Let $\mathfrak{g}$ be a $p$-step nilpotent Lie algebra, $M$ a $q$-step nilpotent $\mathfrak{g}$-module and
\[ 0 \rightarrow \mathfrak{n} \rightarrow \mathfrak{f}_{n,r} \xrightarrow{\pi} \mathfrak{g}\rightarrow 0 \]
a free $r$-step nilpotent extension of $\mathfrak{g}$, then
\[ F_r\mathrm{H}^2(\mathfrak{g},M) = \mathrm{Ker} \ \Big( \pi^2: \mathrm{H}^2(\mathfrak{g},M) \rightarrow \mathrm{H}^2(\mathfrak{f}_{n,r},M)\Big)\]
for every $r \in \mathbb{Z}$ such that $\max\{p,q\} \leq r \leq p+q$.
\end{theorem}
\begin{proof} To prove this we will use the explicit description of the bijective correspondence between $\mathrm{H}^2(\mathfrak{g},M)$ and $\mathrm{Ext}(\mathfrak{g},M)$, as described in the preliminaries. Fix some integer $r$ such that $\max\{p,q\} \leq r \leq p+q$. Now choose an element $\overline{f} \in F_r\mathrm{H}^2(\mathfrak{g},M)$. By Proposition \ref{prop: vergne}, the cocycle $f$ corresponds to an extension $ 0 \rightarrow M \rightarrow \mathfrak{e} \xrightarrow{p} \mathfrak{g} \rightarrow 0$ of $\mathfrak{g}$ by $M$ such that $\mathfrak{e}$ is nilpotent of class at most $r$. Using the pullback construction, we obtain a commutative diagram
\[ \xymatrix{  0 \ar[r] & M \ar[r] \ar[d]^{\mathrm{Id}} & \mathfrak{e}' \ar[r] \ar[d]^{\theta} & \mathfrak{f}_{n,r} \ar[d]^{\pi} \ar[r] & 0  \\
  0 \ar[r] & M \ar[r]  & \mathfrak{e} \ar[r]^{p}  & \mathfrak{g} \ar[r] & 0.} \]
Here, the top sequence is an extension of $\mathfrak{f}_{n,r}$ by $M$, where $M$ becomes a $\mathfrak{f}_{n,r}$-module via $\pi$. By the naturality of the exact sequence of low degree terms coming from the Hochschild-Serre spectral sequence, we obtain a commutative diagram
\[ \xymatrix{\mathrm{Hom}_{\mathfrak{f}_{n,r}}(M,M) \ar[r]^{d_{\mathfrak{f}_{n,p}}} & \mathrm{H}^2(\mathfrak{f}_{n,r},M) \\
\mathrm{Hom}_{\mathfrak{g}}(M,M) \ar[u] \ar[r]^{d_{\mathfrak{g}}} & \mathrm{H}^2(\mathfrak{g},M). \ar[u]^{\pi^2} } \]
Following $\mathrm{Id}_M \in \mathrm{Hom}_{\mathfrak{g}}(M,M)$ through the diagram, we see that $\pi^2(\overline{f})$ is mapped to the element in $\mathrm{H}^2(\mathfrak{f}_{n,r},M)$ corresponding to the equivalence class of $0 \rightarrow M \rightarrow \mathfrak{e}' \rightarrow \mathfrak{f}_{n,r} \rightarrow 0$. Note that it follows from the pullback construction that $\mathfrak{e}'$ is nilpotent of class at most $r$. Therefore, the universal property of $\mathfrak{f}_{n,r}$ implies that $0 \rightarrow M \rightarrow \mathfrak{e}' \rightarrow \mathfrak{f}_{n,r} \rightarrow 0$ splits. Since the class of split extensions of $\mathfrak{f}_{n,r}$ by $M$ corresponds to the zero element in $\mathrm{H}^2(\mathfrak{f}_{n,r},M)$, this shows that $\pi^2(\overline{f})=0$. We have proven that $F_r\mathrm{H}^2(\mathfrak{g},M) \subseteq \mathrm{Ker} \ \pi^2$. \\
Now suppose we have an element $\overline{f} \in \mathrm{H}^2(\mathfrak{g},M)$ such that $\pi^2(\overline{f})=0$. Considering again the two commutative diagrams above, we see that this implies that $0 \rightarrow M \rightarrow \mathfrak{e}' \rightarrow \mathfrak{f}_{n,r} \rightarrow 0$ is split, so $\mathfrak{e} \cong M \rtimes \mathfrak{f}_{n,r}$. Since $M$ is a nilpotent $\mathfrak{f}_{n,r}$-module of class $q$, it follows from Proposition \ref{prop: bound nilp class} that $\mathfrak{e}'$ is nilpotent of class $\max\{q,r\}$. But  $q \leq r$, so $\mathfrak{e}'$ is nilpotent of class $r$. Since one can easily verify that $\theta$ is surjective, $\mathfrak{e}$ is nilpotent of class at most $r$. Hence, by Proposition \ref{prop: vergne}, $\overline{f} \in  F_r\mathrm{H}^2(\mathfrak{g},M)$. This proves that $F_r\mathrm{H}^2(\mathfrak{g},M) \supseteq \mathrm{Ker} \ \pi^2$.
\end{proof}
\begin{corollary} \label{cor: main cor} Let $\mathfrak{g}$ be a $p$-step nilpotent Lie algebra, $M$ a $q$-step nilpotent $\mathfrak{g}$-module and
\[ 0 \rightarrow \mathfrak{n} \rightarrow \mathfrak{f}_{n,r} \xrightarrow{\pi} \mathfrak{g}\rightarrow 0 \]
a free $r$-step nilpotent extension of $\mathfrak{g}$. Then, for every $r \in \mathbb{Z}$ such that $\max\{p,q\} \leq r \leq p+q$, we have an exact sequence
\[ 0 \rightarrow \mathrm{H}^1(\mathfrak{g},M) \rightarrow \mathrm{H}^1(\mathfrak{f}_{n,r},M) \rightarrow \mathrm{Hom}_{\mathfrak{g}}(\frac{\mathfrak{n}}{[\mathfrak{n},\mathfrak{n}]},M) \xrightarrow{d} F_r\mathrm{H}^2(\mathfrak{g},M) \rightarrow 0 \]
such that $d(f)=\overline{f \circ \varphi_r}$,
 where $\varphi_r$ is a $2$-cocycle corresponding to
\[ 0 \rightarrow \frac{\mathfrak{n}}{[\mathfrak{n},\mathfrak{n}]} \rightarrow \frac{\mathfrak{f}_{n,r}}{[\mathfrak{n},\mathfrak{n}]} \xrightarrow{\pi} \mathfrak{g}\rightarrow 0. \]
\end{corollary}
\begin{proof}
Consider the Hochschild-Serre spectral sequence $(E_r(M),d_r)$ associated to the given free $r$-step nilpotent extension of $\mathfrak{g}$. Because the image of $d^{0,1}_2$ is exactly the kernel of $\pi^2: \mathrm{H}^2(\mathfrak{g},M) \rightarrow \mathrm{H}^2(\mathfrak{f}_{n,r},M)$, the theorem implies that the standard exact sequence of low degree term yields the wanted exact sequence. Now let us examine the map $d:=d^{0,1}_2$.  \\
Take $f \in \mathrm{Hom}_{\mathfrak{g}}(\frac{\mathfrak{n}}{[\mathfrak{n},\mathfrak{n}]},M)$ and let $(E_r(\frac{\mathfrak{n}}{[\mathfrak{n},\mathfrak{n}]}),d_r)$ be the Hochschild-Serre spectral sequence associated to the given free $r$-step nilpotent extension of $\mathfrak{g}$ with coefficients in $\frac{\mathfrak{n}}{[\mathfrak{n},\mathfrak{n}]}$. The coefficient map $f: \frac{\mathfrak{n}}{[\mathfrak{n},\mathfrak{n}]} \rightarrow M$ induces a spectral sequence morphism $E_{r}(\frac{\mathfrak{n}}{[\mathfrak{n},\mathfrak{n}]}) \rightarrow E_r(M)$. Hence, we have a commutative diagram
\[ \xymatrix{\mathrm{Hom}_{\mathfrak{g}}(\frac{\mathfrak{n}}{[\mathfrak{n},\mathfrak{n}]},M) \ar[r]^{d} & \mathrm{H}^2(\mathfrak{g},M) \\
\mathrm{Hom}_{\mathfrak{g}}(\frac{\mathfrak{n}}{[\mathfrak{n},\mathfrak{n}]},\frac{\mathfrak{n}}{[\mathfrak{n},\mathfrak{n}]}) \ar[u]^{f \circ \ldots } \ar[r]^{d} & \mathrm{H}^2(\mathfrak{g},\frac{\mathfrak{n}}{[\mathfrak{n},\mathfrak{n}]}). \ar[u]^{f^{2}} } \]
One can check that $d(\mathrm{Id}_{\frac{\mathfrak{n}}{[\mathfrak{n},\mathfrak{n}]}})$ equals $\overline{\varphi_r}$, where $\overline{\varphi_r}$ corresponds to the equivalence class of $0 \rightarrow \frac{\mathfrak{n}}{[\mathfrak{n},\mathfrak{n}]} \rightarrow \frac{\mathfrak{f}_{n,r}}{[\mathfrak{n},\mathfrak{n}]} \xrightarrow{\pi} \mathfrak{g}\rightarrow 0$. Hence, following $\mathrm{Id}_{\frac{\mathfrak{n}}{[\mathfrak{n},\mathfrak{n}]}}$ through the diagram, we see that $d(f)=f^{2}(\overline{\varphi_r})=\overline{f \circ \varphi_r}$. This finishes the proof.
\end{proof}
If we restrict to the case of central extensions of $\mathfrak{g}$ with one-dimensional kernel, our filtrations reduce to
\[ F_p\mathrm{Ext}(\mathfrak{g},k) \subseteq F_{p+1}\mathrm{Ext}(\mathfrak{g},k)=\mathrm{Ext}(\mathfrak{g},k)\]
and
\[ F_p\mathrm{H}^2(\mathfrak{g},k) \subseteq F_{p+1}\mathrm{H}^2(\mathfrak{g},k)=\mathrm{H}^2(\mathfrak{g},k), \]
when $\mathfrak{g}$ is $p$-step nilpotent and $k$ is viewed as a trivial $\mathfrak{g}$-module. representatives of elements in $F_p\mathrm{Ext}(\mathfrak{g},k)$ are called central extensions of class $p$, whereas representatives of elements in $\mathrm{Ext}(\mathfrak{g},k)$ that are not contained in $F_p\mathrm{Ext}(\mathfrak{g},k)$ are called central extensions of class $p+1$. \\ \\
\indent Using Corollary \ref{cor: main cor}, we obtain an description of $F_p\mathrm{H}^2(\mathfrak{g},k)$ and $\mathrm{H}^2(\mathfrak{g},k)$.
\begin{corollary} \label{cor: central dimension exprs} Let $\mathfrak{g}$ be a $p$-step nilpotent Lie algebra. If we consider trivial coefficients $k$, then the map $d$ from Corollary \ref{cor: main cor} is an isomorphism, for $r \in \{p,p+1\}$.
In particular, if
\[ 0 \rightarrow \mathfrak{n} \rightarrow \mathfrak{f}_{n,p} \xrightarrow{\pi} \mathfrak{g} \rightarrow 0 \]
is a free $p$-step nilpotent extension, then
\begin{eqnarray*}
|F_p\mathrm{H}^2(\mathfrak{g},k)| &=& |\frac{\mathfrak{n}}{[\mathfrak{f}_{n,p},\mathfrak{n}]}|;\\
|\mathrm{H}^2(\mathfrak{g},k)|&=&  |\frac{\mathfrak{n}}{[\mathfrak{f}_{n,p},\mathfrak{n}]}| + b_2(\mathfrak{f}_{n,p})-|[i(\mathfrak{n}),\mathfrak{f}_{n,p+1}]\cap \mathfrak{f}_{n,p+1}^{p+1}|,
\end{eqnarray*}
where $i: \mathfrak{f}_{n,p} \rightarrow \mathfrak{f}_{n,p+1}$ is the canonical vector space inclusion.
\end{corollary}
\begin{proof}
Suppose we have an $r$-step nilpotent extension of $\mathfrak{g}$, with kernel $\mathfrak{n}$. Because $\mathfrak{n} \subset [\mathfrak{f}_{n,r},\mathfrak{f}_{n,r}]$ and the coefficients are trivial, it follows easily that $\mathrm{H}^1(\mathfrak{f}_{n,r},k) \rightarrow \mathrm{H}^1(\mathfrak{n},k)^{\mathfrak{g}}= \mathrm{Hom}_{\mathfrak{g}}(\frac{\mathfrak{n}}{[\mathfrak{n},\mathfrak{n}]},k)$ is the zero map. By exactness of the sequence in Corollary \ref{cor: main cor}, $\mathrm{Hom}_{\mathfrak{g}}(\frac{\mathfrak{n}}{[\mathfrak{n},\mathfrak{n}]},k) \xrightarrow{d} F_r\mathrm{H}^2(\mathfrak{g},k)$ is an isomorphism. \\
Now consider the given $p$-step nilpotent extension of $\mathfrak{g}$.
We know that $\mathrm{Hom}_{\mathfrak{g}}(\frac{\mathfrak{n}}{[\mathfrak{n},\mathfrak{n}]},k) \cong F_p\mathrm{H}^2(\mathfrak{g},k)$. Using the fact that $\mathfrak{f}_{n,p}$ acts trivially on $k$, it is not difficult to check that $\mathrm{Hom}_{\mathfrak{g}}(\frac{\mathfrak{n}}{[\mathfrak{n},\mathfrak{n}]},k)=\mathrm{Hom}_{\mathfrak{f}_{n,p}}(\frac{\mathfrak{n}}{[\mathfrak{n},\mathfrak{n}]},k)\cong \frac{\mathfrak{n}}{[\mathfrak{f}_{n,p},\mathfrak{n}]}$ as vector spaces. So the first equality is proven. \\
Now consider the free $(p+1)$-step nilpotent Lie algebra $\mathfrak{f}_{n,p+1}$ together with the canonical projection $p: \mathfrak{f}_{n,p+1} \rightarrow \mathfrak{f}_{n,p+1}/\mathfrak{f}_{n,p+1}^{p+1}=\mathfrak{f}_{n,p}$ and its canonical vector space splitting $i: \mathfrak{f}_{n,p} \rightarrow \mathfrak{f}_{n,p+1}$. Then, by definition, $0 \rightarrow \mathfrak{f}_{n,p+1}^{p+1}+i(\mathfrak{n}) \rightarrow \mathfrak{f}_{n,p+1} \rightarrow \mathfrak{g} \rightarrow 0$ is a free $(p+1)$-step nilpotent extension of $\mathfrak{g}$. Hence, once again we have $\mathrm{Hom}_{\mathfrak{g}}(\frac{i(\mathfrak{n})+\mathfrak{f}_{n,p+1}^{p+1}}{[i(\mathfrak{n})+\mathfrak{f}_{n,p+1}^{p+1},i(\mathfrak{n})+\mathfrak{f}_{n,p+1}^{p+1}]},k) \cong \mathrm{H}^2(\mathfrak{g},k)$, and $\mathrm{Hom}_{\mathfrak{g}}(\frac{i(\mathfrak{n})+\mathfrak{f}_{n,p+1}^{p+1}}{[i(\mathfrak{n})+\mathfrak{f}_{n,p+1}^{p+1},i(\mathfrak{n})+\mathfrak{f}_{n,p+1}^{p+1}]},k) \cong \frac{i(\mathfrak{n})+\mathfrak{f}_{n,p+1}^{p+1}}{[\mathfrak{f}_{n,p+1},i(\mathfrak{n})+\mathfrak{f}_{n,p+1}^{p+1}]}$ as vector spaces. Furthermore, one can verify that there is a vector space isomorphism
\[ \frac{i(\mathfrak{n})+\mathfrak{f}_{n,p+1}^{p+1}}{[\mathfrak{f}_{n,p+1},i(\mathfrak{n})+\mathfrak{f}_{n,p+1}^{p+1}]} \cong \frac{\mathfrak{n}}{[\mathfrak{f}_{n,p},\mathfrak{n}]} \oplus \frac{\mathfrak{f}_{n,p+1}^{p+1}}{[i(\mathfrak{n}),\mathfrak{f}_{n,p+1}]\cap \mathfrak{f}_{n,p+1}^{p+1}}.
\]
Since $|\mathfrak{f}_{n,p+1}^{p+1}|=b_2(\mathfrak{f}_{n,p})$, the second equality follows.
\end{proof}
We also obtain the following criterium for the existence of central extensions of class $p+1$.
\begin{corollary} \label{cor: central extension}
Let $\mathfrak{g}$ be a $p$-step nilpotent Lie algebra with free $p$-step nilpotent extension
\[ 0 \rightarrow \mathfrak{n} \rightarrow \mathfrak{f}_{n,p} \xrightarrow{\pi} \mathfrak{g} \rightarrow 0. \]
Then $\mathfrak{g}$ admits a central extension of class $p+1$ if and only if
\[  |\frac{\mathfrak{n}}{[\mathfrak{f}_{n,p},\mathfrak{n}]}| < b_2(\mathfrak{g}). \]
In particular, if $\mathfrak{n}$ is generated by a single element $X \in [\mathfrak{f}_{n,p},\mathfrak{f}_{n,p}]$, then $\mathfrak{g}$ admits a central extension of class $p+1$.
\end{corollary}
\begin{proof}
The first statement is immediate from the previous corollary and the definitions of $F_p\mathrm{Ext}(\mathfrak{g},k)$ and $F_{p+1}\mathrm{Ext}(\mathfrak{g},k)=\mathrm{Ext}(\mathfrak{g},k)$. \\
Now suppose that $\mathfrak{n}$ is generated by a single element $X \in H_r(n) \subset  [\mathfrak{f}_{n,p},\mathfrak{f}_{n,p}]$.
Let $\gamma_r(\mathfrak{f}_{n,p},X)$ be the vector space spanned by $X$ and define $\gamma_i(\mathfrak{f}_{n,p},X)=[\mathfrak{f}_{n,p},\gamma_{i-1}(\mathfrak{f}_{n,p},X)]$ for $i \in \{r+1,\ldots, p\}$. Then it is clear that $\gamma_i(\mathfrak{f}_{n,p},X) \subset H_i(n)$ for all $i \in \{r,\ldots, p\}$ and $\mathfrak{n}=\gamma_1(\mathfrak{f}_{n,p},X) + \gamma_2(\mathfrak{f}_{n,p},X) + \ldots + \gamma_p(\mathfrak{f}_{n,p},X)$. This implies that $[\mathfrak{f}_{n,p},\mathfrak{n}]= \gamma_2(\mathfrak{f}_{n,p},X) + \ldots + \gamma_p(\mathfrak{f}_{n,p},X)$, hence $|\frac{\mathfrak{n}}{[\mathfrak{f}_{n,p},\mathfrak{n}]}|=|\gamma_r(\mathfrak{f}_{n,p},X)|=1$. Since nilpotent Lie algebras $\mathfrak{g}$ always satisfy $b_2(\mathfrak{g})\geq 2$, we conclude that $\mathfrak{g}$ has a central extension of class $p+1$.
\end{proof}
We conclude this section with a formula for $|F_2\mathrm{H}^2(\mathfrak{g},k)|$ and $|F_2\mathrm{H}^2(\mathfrak{g},\mathfrak{g})|$ in the case where $\mathfrak{g}$ is a $2$-step nilpotent Lie algebra.
\begin{corollary} Let $\mathfrak{g}$ be a $2$-step nilpotent Lie algebra with $b_1(\mathfrak{g})=n$, then
\begin{eqnarray*}
|F_2\mathrm{H}^2(\mathfrak{g},k)| & = & \binom{n}{2}+n - |\mathfrak{g}|; \\
|F_2\mathrm{H}^2(\mathfrak{g},\mathfrak{g})| &=& \binom{n+1}{2}|Z(\mathfrak{g})|- n|\mathfrak{g}| - |\mathfrak{g}||Z(\mathfrak{g})|+|\mathrm{Der}(\mathfrak{g})|.
\end{eqnarray*}
Here, $Z(\mathfrak{g})$ is the center of $\mathfrak{g}$ and $\mathrm{Der}(\mathfrak{g})$ is the derivation algebra of $\mathfrak{g}$.
\end{corollary}
\begin{proof} Choose a $2$-step nilpotent extension $0 \rightarrow \mathfrak{n} \rightarrow \mathfrak{f}_{n,2} \xrightarrow{\pi} \mathfrak{g} \rightarrow 0$ of $\mathfrak{g}$.
Corollary \ref{cor: central dimension exprs} tells us that $|F_2\mathrm{H}^2(\mathfrak{g},k)|=|\frac{\mathfrak{n}}{[\mathfrak{f}_{n,2},\mathfrak{n}]}|$. Because $[\mathfrak{f}_{n,2},\mathfrak{n}]=0$ and $|\mathfrak{f}_{n,2}|=\binom{n}{2}+n$, it follows that $|F_2\mathrm{H}^2(\mathfrak{g},k)|= \binom{n}{2}+n -|\mathfrak{g}|$. \\
Using Corollary \ref{cor: main cor}, we obtain an exact sequence
\[ 0 \rightarrow \mathrm{H}^1(\mathfrak{g},\mathfrak{g}) \rightarrow \mathrm{H}^1(\mathfrak{f}_{n,r},\mathfrak{g}) \rightarrow \mathrm{Hom}_{\mathfrak{g}}(\mathfrak{n},\mathfrak{g}) \xrightarrow{d} F_2\mathrm{H}^2(\mathfrak{g},\mathfrak{g}) \rightarrow 0. \]
This implies that $|F_2\mathrm{H}^2(\mathfrak{g},\mathfrak{g})|= |\mathrm{Hom}_{\mathfrak{g}}(\mathfrak{n},\mathfrak{g})|-|\mathrm{H}^1(\mathfrak{f}_{n,2},\mathfrak{g})|+|\mathrm{H}^1(\mathfrak{g},\mathfrak{g})|$.
Since $\mathfrak{g}$ acts trivially on $\mathfrak{n}$, one easily verifies that $|\mathrm{Hom}_{\mathfrak{g}}(\mathfrak{n},\mathfrak{g})|=|\mathfrak{n}||Z(\mathfrak{g})|=(\binom{n}{2}+n-|\mathfrak{g}|)|Z(\mathfrak{g})|$. Also, it is well-known that $\mathrm{H}^1(\mathfrak{g},\mathfrak{g})\cong \mathrm{Der}(\mathfrak{g})/\mathrm{Inn}(\mathfrak{g})$, where $\mathrm{Inn}(\mathfrak{g})$ is the space of inner derivations of $\mathfrak{g}$. Since $\mathrm{Inn}(\mathfrak{g})\cong \mathfrak{g}/Z(\mathfrak{g})$, we conclude that $|\mathrm{H}^1(\mathfrak{g},\mathfrak{g})|=|\mathrm{Der}(\mathfrak{g})|-|\mathfrak{g}|+|Z(\mathfrak{g})|$. It remains to determine $|\mathrm{H}^1(\mathfrak{f}_{n,r},\mathfrak{g})|$. Considering $0 \rightarrow \mathfrak{n} \rightarrow \mathfrak{f}_{n,2} \xrightarrow{\pi} \mathfrak{g} \rightarrow 0$ as a short exact sequence of $\mathfrak{f}_{n,2}$-modules, we obtain a long exact cohomology sequence
\[ 0 \rightarrow \mathfrak{n} \rightarrow Z(\mathfrak{f}_{n,2}) \rightarrow Z(\mathfrak{g}) \rightarrow \mathrm{Hom}_k(\frac{\mathfrak{f}_{n,2}}{[\mathfrak{f}_{n,2},\mathfrak{f}_{n,2}]},\mathfrak{n}) \rightarrow \mathrm{H}^1(\mathfrak{f}_{n,2},\mathfrak{f}_{n,2}) \rightarrow \mathrm{H}^1(\mathfrak{f}_{n,2},\mathfrak{g}). \]
We claim that the map $\mathrm{H}^1(\mathfrak{f}_{n,2},\mathfrak{f}_{n,2}) \rightarrow \mathrm{H}^1(\mathfrak{f}_{n,2},\mathfrak{g})$ is surjective. To prove this, first recall that if $\{x_1,\ldots,x_n\}$ are the generators of $\mathfrak{f}_{n,2}$, then any linear map from $<x_1,\ldots,x_n>$ to a $\mathfrak{f}_{n,2}$-module $M$ can be uniquely extended to a derivation of $\mathfrak{f}_{n,2}$ into $M$. Now choose a derivation $\varphi \in \mathrm{Der}(\mathfrak{f}_{n,2},\mathfrak{g})$ and a vector space splitting $s: \mathfrak{g} \rightarrow \mathfrak{f}_{n,2}$ of $\pi: \mathfrak{f}_{n,2} \rightarrow \mathfrak{g}$. Then the linear map from  $<x_1,\ldots,x_n>$ to $\mathfrak{f}_{n,2}$ that takes $x_i$ to $s(\varphi(x_i))$ can be extended to a derivation $\psi \in \mathrm{Der}(\mathfrak{f}_{n,2})$. Now consider $\overline{\psi} \in \mathrm{H}^1(\mathfrak{f}_{n,2},\mathfrak{f}_{n,2})=\mathrm{Der}(\mathfrak{f}_{n,2})/\mathrm{Inn}(\mathfrak{f}_{n,2})$. Then, under $\mathrm{H}^1(\mathfrak{f}_{n,2},\mathfrak{f}_{n,2}) \rightarrow \mathrm{H}^1(\mathfrak{f}_{n,2},\mathfrak{g})$, $\overline{\psi}$ is mapped to $\overline{\pi \circ \psi}$. Since  $\pi(\psi(x_i))=\varphi(x_i)$ for all $i$, and a derivation of a nilpotent Lie algebra is completely determined by its value on the generators, we conclude that $\pi \circ \psi=\varphi$ and hence $\mathrm{H}^1(\mathfrak{f}_{n,2},\mathfrak{f}_{n,2}) \rightarrow \mathrm{H}^1(\mathfrak{f}_{n,2},\mathfrak{g}): \overline{\psi} \mapsto \overline{\varphi}$. Since $\overline{\varphi}$ is chosen at random, we conclude that $\mathrm{H}^1(\mathfrak{f}_{n,2},\mathfrak{f}_{n,2}) \rightarrow \mathrm{H}^1(\mathfrak{f}_{n,2},\mathfrak{g})$ is surjective. This implies that $|\mathrm{H}^1(\mathfrak{f}_{n,r},\mathfrak{g})|=|\mathrm{H}^1(\mathfrak{f}_{n,r},\mathfrak{f}_{n,2})|-|\mathrm{Hom}_k(\frac{\mathfrak{f}_{n,2}}{[\mathfrak{f}_{n,2},\mathfrak{f}_{n,2}]},\mathfrak{n})|
+|Z(\mathfrak{g})|-|Z(\mathfrak{f}_{n,2})|+|\mathfrak{n}|$. Since $|\mathfrak{n}|=\binom{n}{2}+n-|\mathfrak{g}|$, $|H^1(\mathfrak{f}_{n,2},\mathfrak{f}_{n,2})|=n\Big( \binom{n}{2}+n-1\Big)$ and $|\mathrm{Hom}_k(\frac{\mathfrak{f}_{n,2}}{[\mathfrak{f}_{n,2},\mathfrak{f}_{n,2}]},\mathfrak{n})|=n|\mathfrak{n}|$, one can verify that $|F_2\mathrm{H}^2(\mathfrak{g},\mathfrak{g})| = \binom{n+1}{2}|Z(\mathfrak{g})|- n|\mathfrak{g}| - |\mathfrak{g}||Z(\mathfrak{g})|+|\mathrm{Der}(\mathfrak{g})|$.
\end{proof}
\begin{remark} \rm When $\mathfrak{g}$ is a $p$-step nilpotent Lie algebra of depth $p$, one could mimic the proof above to obtain formulas for $|F_p\mathrm{H}^2(\mathfrak{g},k)|$ and $|F_p\mathrm{H}^2(\mathfrak{g},\mathfrak{g})|$. The key observation is that the kernel of the free $p$-step nilpotent extension is central. This formula will contain the dimensions of $\mathfrak{f}_{n,p}$, $Z(\mathfrak{f}_{n,p})$ and $\mathrm{Der}(\mathfrak{f}_{n,p})$, which can all be computed using the M\"{o}bius function.
\end{remark}
\section{Bounds for $b_2(\mathfrak{g})$}
In this section we will, given a $p$-step nilpotent Lie algebra $\mathfrak{g}$, derive some upper and lower bounds for $b_2(\mathfrak{g})$. We start with two technical lemmas.
\begin{lemma} \label{lemma : cocycle lemma}
Let $\mathfrak{g}$ be a finite dimensional Lie algebra over $k$. Take $\overline{\mu} \in \mathrm{H}^2(\mathfrak{g},k)$, $x \in \mathrm{Z}(\mathfrak{g})$ and $y \in [\mathfrak{g},\mathfrak{g}]$, then $\mu(x\wedge y)=0.$ Moreover, if $\mathfrak{g}$ is $p$-step nilpotent ($p\geq 2$) then $\mu(\mathfrak{g}^{p}\wedge \mathfrak{g}^{p} )=0$ and the induced map $\mathrm{H}^2(\mathfrak{g},k) \rightarrow \mathrm{H}^2(\mathfrak{g}^{p},k)$ is zero.
\end{lemma}
\begin{proof}
By definition we have $y=\sum_{j=1}^k[a_j,b_j]$, for some $a_j,b_j \in \mathfrak{g}$, and $\mu \in \mathrm{Hom}_k(\Lambda^2(\mathfrak{g}),k)$ with $d^2(\mu)=0$. Here $d^2$ is the second differential in the Chevalley-Eilenberg complex of $\mathfrak{g}$. Since  $d^2(\mu)=0$, we have
$0  =  d^2\mu(\sum_{j=1}^ka_j\wedge b_j \wedge x)
 =  \sum_{j=1}^k \Big( -\mu([a_j,b_j]\wedge x) + \mu([a_j,x]\wedge b_j) - \mu([b_j,x]\wedge a_j)\Big)
 =  \mu (x \wedge (\sum_{j=1}^k[a_j,b_j])) = \mu(x\wedge y)$.
Furthermore, if  $\mathfrak{g}$ is $p$-step nilpotent ($p\geq 2$) then $\mathfrak{g}^{p} \subset \mathrm{Z}(\mathfrak{g})\cap [\mathfrak{g},\mathfrak{g}]$. So the remaining statements follow readily.
\end{proof}

\begin{lemma} \label{lemma: betti number lemma} Consider the extension $0 \rightarrow \mathfrak{n} \rightarrow \mathfrak{g} \rightarrow \mathfrak{h} \rightarrow 0$ and assume that $\mathfrak{n} \subset Z(\mathfrak{g}) \cap \mathfrak{g}^{2}$.
Then $b_1(\mathfrak{g})=b_1(\mathfrak{h})$ and $b_2(\mathfrak{g}) = b_2(\mathfrak{h})-|\mathfrak{n}|+ |E^{1,1}_{\infty}|$. Moreover, $E^{1,1}_{\infty}$ equals the cokernel of $\mathrm{H}^2(\mathfrak{h},k) \rightarrow \mathrm{H}^2(\mathfrak{g},k)$.
\end{lemma}
\begin{proof} Denote by $(E_r,d_r)$ the Hochschild-Serre spectral sequence associated to the given extension of Lie algebras. Recall that the image of the map $\mathrm{H}^2(\mathfrak{h},k) \rightarrow \mathrm{H}^2(\mathfrak{g},k)$ equals $\mathrm{E}_{\infty}^{2,0}$ and the image of the map $\mathrm{H}^1(\mathfrak{g},k) \rightarrow \mathrm{H}^1(\mathfrak{n},k)$ is isomorphic to  $\mathrm{E}_{\infty}^{0,1}$. Hence, the cokernel of $\mathrm{H}^2(\mathfrak{h},k) \rightarrow \mathrm{H}^2(\mathfrak{g},k)$ equals $\mathrm{E}_{\infty}^{0,2}\oplus \mathrm{E}_{\infty}^{1,1}$. But Lemma \ref{lemma : cocycle lemma} implies that $\mathrm{H}^2(\mathfrak{g},k) \rightarrow \mathrm{H}^2(\mathfrak{n},k)$ is the zero map. Since the image of this map is isomorphic to $E_{\infty}^{0,2}$, we conclude that $E_{\infty}^{0,2}=0$, so the cokernel of $\mathrm{H}^2(\mathfrak{h},k) \rightarrow \mathrm{H}^2(\mathfrak{g},k)$ equals $\mathrm{E}_{\infty}^{1,1}$. \\
Because $\mathfrak{n}\subset \mathfrak{g}^2$, it follows that the induced map $\mathrm{H}^1(\mathfrak{g},k) \rightarrow \mathrm{H}^1(\mathfrak{n},k)$ is zero. Therefore $\mathrm{E}_{\infty}^{0,1}=0$, which  implies that $d_2^{0,1}$ is injective. Since this differential is the only one landing in the $(2,0)$-spot, we conclude that $\mathrm{E}_{\infty}^{2,0}=\mathrm{E}_{2}^{2,0}/(\mathrm{Im} \ d_2^{0,1})$. Furthermore, the spectral sequence converges to $\mathrm{H}^{\ast}(\mathfrak{g},k)$, so we have $ \mathrm{H}^{1}(\mathfrak{g},k) =\mathrm{E}_{\infty}^{0,1}\oplus \mathrm{E}_{\infty}^{1,0} = \mathrm{H}^{1}(\mathfrak{h},k)$ and
$\mathrm{H}^{2}(\mathfrak{g},k) = \mathrm{E}_{\infty}^{0,2}\oplus \mathrm{E}_{\infty}^{1,1} \oplus \mathrm{E}_{\infty}^{2,0} =\mathrm{E}_{\infty}^{1,1} \oplus \mathrm{H}^{2}(\mathfrak{h},k)/ (\mathrm{Im} \ d_2^{0,1})$. The equalities now readily follow.
\end{proof}
The following proposition provides an upper bound for the second Betti number. This bound can also be found in \cite{Yankosky}.
Here, we give an alternative proof using the previous two lemmas.
\begin{proposition} \label{prop: upperbound second betti}If $\mathfrak{g}$ is an $m$-dimensional $p$-step nilpotent Lie algebra of type $(n_1,n_2,\ldots,n_p)$, then
\[ b_2(\mathfrak{g}) \leq \binom{n_1}{2}+(m-n_1)(n_1-1). \]
In particular, if $\mathfrak{g}$ is an $m$-dimensional nilpotent Lie algebra with $b_1(\mathfrak{g})=2$, then $b_2(\mathfrak{g}) \leq m-1$.
\end{proposition}
\begin{proof} We will use induction on $p$. If $p=1$, then $m=n_1$ and $b_2(\mathfrak{g})=\binom{m}{2}$, so in this case the statement is trivial. Now assume the statement is true for $(p-1)$-step nilpotent Lie algebras ($p\geq 2$) and suppose that $\mathfrak{g}$ is $p$-step nilpotent. Considering the extension
\[ 0 \rightarrow \mathfrak{g}^{p} \rightarrow \mathfrak{g} \xrightarrow{\pi} \mathfrak{g}/\mathfrak{g}^{p} \rightarrow 0, \]
we deduce from Lemma \ref{lemma: betti number lemma} that $ b_2(\mathfrak{g})=b_2(\mathfrak{g}/\mathfrak{g}^{p-1})-n_p+|E^{1,1}_{\infty}|$. Hence, using the induction hypothesis and the fact that $| \mathrm{E}_{\infty}^{1,1} | \leq | \mathrm{H}^1(\mathfrak{g}/\mathfrak{g}^{p},\mathrm{H}^1(\mathfrak{g}^{p},k))|=n_1n_p$, we find
$b_2(\mathfrak{g})  \leq b_2(\mathfrak{g}/\mathfrak{g}^{p})-n_p + n_1n_p
\leq  \binom{n_1}{2}+(m-n_p-n_1)(n_1-1) -n_p + n_1n_p = \binom{n_1}{2}+(m-n_1)(n_1-1)$.
This completes the induction and proves the proposition.
\end{proof}
We will now derive some upper and lower bounds for $b_2(\mathfrak{g})$, using the results from the previous section.
\begin{lemma} \label{lemma: lowerbound1}
Let $\mathfrak{g}$ be a $p$-step nilpotent Lie algebra of type $(n_1,\ldots,n_p)$, depth $d$ and with free $p$-step nilpotent extension
\[ 0 \rightarrow \mathfrak{n} \rightarrow \mathfrak{f}_{n_1,p} \xrightarrow{\pi} \mathfrak{g} \rightarrow 0. \] Then
 $b_2(\mathfrak{f}_{n_1,d-1})-n_{d} \leq |F_p\mathrm{H}^2(\mathfrak{g},k)|$ and equality holds if and only if $[\mathfrak{f}_{n_1,p},\mathfrak{n}] = \mathfrak{n}\cap \mathfrak{f}^{d+1}_{n_1,p}$. In particular, equality holds if $\mathfrak{n}$ is generated by Lie monomials of degree $d$ (for example when $d=p$).
\end{lemma}
\begin{proof}
The surjection $\pi: \mathfrak{f}^d_{n_1,p} \rightarrow \mathfrak{g}^d$ induces a surjection $\overline{\pi}: \frac{\mathfrak{f}^d_{n_1,p}}{\mathfrak{f}^{d+1}_{n_1,p}} \rightarrow \frac{\mathfrak{g}^{d}}{\mathfrak{g}^{d+1}}$. Since $\mathfrak{g}$ is of depth $d$, one can verify that $\ker \overline{\pi}\cong \frac{\mathfrak{n}}{\mathfrak{n}\cap \mathfrak{f}^{d+1}_{n_1,p}}$. Clearly $[\mathfrak{f}_{n_1,p},\mathfrak{n}] \subset \mathfrak{n}\cap \mathfrak{f}^{d+1}_{n_1,p}$ and $|\frac{\mathfrak{g}^{d}}{\mathfrak{g}^{d+1}}| + |\ker \overline{\pi}| = |  \frac{\mathfrak{f}^d_{n_1,p}}{\mathfrak{f}^{d+1}_{n_1,p}}|$, so Corollary \ref{cor: central dimension exprs} implies that
\[ |\frac{\mathfrak{f}^d_{n_1,p}}{\mathfrak{f}^{d+1}_{n_1,p}}| - |\frac{\mathfrak{g}^{d}}{\mathfrak{g}^{d+1}}| \leq |\frac{\mathfrak{n}}{[\mathfrak{f}_{n_1,p},\mathfrak{n}]}|=|F_p\mathrm{H}^2(\mathfrak{g},k)| . \]
Since $|\frac{\mathfrak{f}^d_{n_1,p}}{\mathfrak{f}^{d+1}_{n_1,p}}|=b_2(\mathfrak{f}_{n_1,d-1})$ and $|\frac{\mathfrak{g}^{d}}{\mathfrak{g}^{d+1}}|=n_d$, the inequality holds and it follows readily that equality holds if and only if $[\mathfrak{f}_{n_1,p},\mathfrak{n}] = \mathfrak{n}\cap \mathfrak{f}^{d+1}_{n_1,p}$. \\
Now suppose that $\mathfrak{n}$ is an ideal in $\mathfrak{f}_{n_1,p}$ generated by $L$, a set of Lie monomials of degree $d$ (i.e. $L \subset H_{d}(n_1)$ ). Also, let $\gamma_d(\mathfrak{f}_{n_1,p},L)$ be the vector space spanned by $L$ and define $\gamma_i(\mathfrak{f}_{n_1,p},L):=[\mathfrak{f}_{n_1,p},\gamma_{i-1}(\mathfrak{f}_{n_1,p},L)]$ for $i \in \{d+1,\ldots, p\}$. It is clear that $\gamma_i(\mathfrak{f}_{n_1,p},L) \subset \mathfrak{f}_{n_1,p}^i$ for all $i \in \{d,\ldots, p\}$ and $\mathfrak{n}=\gamma_d(\mathfrak{f}_{n_1,p},L) + \gamma_{d+1}(\mathfrak{f}_{n_1,p},L) + \ldots + \gamma_p(\mathfrak{f}_{n_1,p},L)$. It follows that $\mathfrak{n} \cap  \mathfrak{f}^{d+1}_{n_1,p}= \gamma_{d+1}(\mathfrak{f}_{n_1,p},L) + \ldots + \gamma_p(\mathfrak{f}_{n_1,p},L)=[\mathfrak{f}_{n_1,p},\mathfrak{n}]$, which proves the lemma.
\end{proof}
This next lemma gives upper and lower bounds for $|F_{p+1}\mathrm{H}^2(\mathfrak{g},k)|-|F_p\mathrm{H}^2(\mathfrak{g},k)|$.
\begin{lemma} \label{lemma: lowerbound2} If $\mathfrak{g}$ is a $p$-step nilpotent Lie algebra of type $(n_1,\ldots,n_p)$ with free $p$-step nilpotent extension
\[ 0 \rightarrow \mathfrak{n} \rightarrow \mathfrak{f}_{n_1,p} \xrightarrow{\pi} \mathfrak{g} \rightarrow 0, \]
then $b_2(\mathfrak{f}_{n_1,p})-b_2(\mathfrak{f}_{n_1,p-1})+n_p \geq  |F_{p+1}\mathrm{H}^2(\mathfrak{g},k)|-|F_p\mathrm{H}^2(\mathfrak{g},k)|$.
Furthermore, if $\mathfrak{g}$ is of depth $p$, then $|F_{p+1}\mathrm{H}^2(\mathfrak{g},k)|-|F_p\mathrm{H}^2(\mathfrak{g},k)| \geq \max\{0,b_2(\mathfrak{f}_{n_1,p})-n_1b_2(\mathfrak{f}_{n_1,p-1})+n_1n_p\}$.
\end{lemma}
\begin{proof}
First notice that $[i(\mathfrak{n})\cap \mathfrak{f}^p_{n_1,p+1}, \mathfrak{f}_{n_1,p+1}] \subset [i(\mathfrak{n}),\mathfrak{f}_{n_1,p+1}]\cap \mathfrak{f}_{n_1,p+1}^{p+1}$.
Now suppose that $\{v_1,\ldots,v_m\}$ is a basis for $i(\mathfrak{n})\cap \mathfrak{f}^p_{n_1,p+1}$. Then $\{[x_1,v_1],\ldots,[x_1,v_m]\}$ is a linearly independent subset of $[i(\mathfrak{n})\cap \mathfrak{f}^p_{n_1,p+1}, \mathfrak{f}_{n_1,p+1}]$. Indeed, $\sum_{i=1}^m\lambda_i[x_1,v_i]=0$ implies that $[x_1,\sum_{i=1}^m\lambda_i v_i]=0$. Since we are working in a free nilpotent Lie algebra it follows that $\sum_{i=1}^m\lambda_i v_i=0$, so all the $\lambda_i$'s are zero. By noting that $|i(\mathfrak{n})\cap \mathfrak{f}^p_{n_1,p+1}|= |\mathfrak{f}_{n_1,p}^p|-n_p= b_2(\mathfrak{f}_{n_1,p-1})-n_p$, we obtain $b_2(\mathfrak{f}_{n_1,p-1})-n_p \leq |[i(\mathfrak{n}),\mathfrak{f}_{n_1,p+1}]\cap \mathfrak{f}_{n_1,p+1}^{p+1}|$. Therefore, Corollary \ref{cor: central dimension exprs} yields $b_2(\mathfrak{f}_{n_1,p})-b_2(\mathfrak{f}_{n_1,p-1})+n_p \geq  |F_{p+1}\mathrm{H}^2(\mathfrak{g},k)|-|F_p\mathrm{H}^2(\mathfrak{g},k)|$. \\
Now assume that $\mathfrak{g}$ is of depth $p$. In this case $[i(\mathfrak{n})\cap \mathfrak{f}^p_{n_1,p+1}, \mathfrak{f}_{n_1,p+1}] = [i(\mathfrak{n}),\mathfrak{f}_{n_1,p+1}]\cap \mathfrak{f}_{n_1,p+1}^{p+1}$.
Clearly, we also have $[i(\mathfrak{n})\cap \mathfrak{f}^p_{n_1,p+1}, \mathfrak{f}_{n_1,p+1}^{2}]=0$. Hence, it follows that  $[i(\mathfrak{n}),\mathfrak{f}_{n_1,p+1}]\cap \mathfrak{f}_{n_1,p+1}^{p+1}= [i(\mathfrak{n})\cap \mathfrak{f}^p_{n_1,p+1}, H_1(n_1)]$. In general, if $A$ is an $n$-dimensional subspace of a Lie algebra and  $B$ is an $m$-dimensional subspace, we have $| [A,B]|\leq nm$. Since $|H_1(n_1)|=n_1$ and $|i(\mathfrak{n})\cap \mathfrak{f}^p_{n_1,p+1}|= b_2(\mathfrak{f}_{n_1,p-1})-n_p$, we are done.
\end{proof}
Putting all the pieces together, we arrive at the following estimates for the second Betti number.
\begin{theorem} \label{th: betti number estimate}
Let $\mathfrak{g}$ be a $p$-step nilpotent Lie algebra of type $(n_1,n_2,\ldots,n_{p})$, dimension $m$, depth $d$ and with free $p$-step nilpotent extension
\[ 0 \rightarrow \mathfrak{n} \rightarrow \mathfrak{f}_{n_1,p} \xrightarrow{\pi} \mathfrak{g} \rightarrow 0. \]
Define $c:= b_2(\mathfrak{f}_{n_1,d-1})-n_{d} $ and $C:= \binom{n_1}{2}+(m-n_1)(n_1-1)$ . In general, we have
\[ c \leq b_2(\mathfrak{g}) \leq C.\]
More specifically, if $\mathfrak{n}$ is generated by Lie monomials of degree $d$, we have
\[ c \leq b_2(\mathfrak{g}) \leq \min\{c + b_2(\mathfrak{f}_{n_1,p})-b_2(\mathfrak{f}_{n_1,p-1})+n_p,C \} \]
and if $\mathfrak{g}$ is of depth $p$ then
\[ c + \max{\{0,b_2(\mathfrak{f}_{n_1,p})-n_1b_2(\mathfrak{f}_{n_1,p-1})+n_1n_p\}} \leq b_2(\mathfrak{g}) \leq \min\{c + b_2(\mathfrak{f}_{n_1,p})-b_2(\mathfrak{f}_{n_1,p-1})+n_p,C \}. \]
\end{theorem}
\begin{proof}
This follows immediately, by combining Proposition \ref{prop: upperbound second betti}, Lemma \ref{lemma: lowerbound1}, Lemma \ref{lemma: lowerbound2} and Corollary \ref{cor: central dimension exprs}.
\end{proof}
In the special case of a $2$-step nilpotent Lie algebra, we find the following estimates.
\begin{corollary} Let $\mathfrak{g}$ be a $2$-step nilpotent Lie algebra of type $(n_1,n_2)$, then
\[ \binom{n_1}{2}-n_2 + \max{\{0, 2\binom{n_1+1}{3}-n_1\binom{n_1}{2}+n_1n_2\}}\leq b_2(\mathfrak{g}) \leq \binom{n_1}{2}+n_2(n_1-1). \]
\end{corollary}
\begin{proof} This is immediate from the final statement of the previous theorem, since a $2$-step nilpotent Lie algebra is of depth $2$, $b_2(\mathfrak{f}_{n_1,1})=\binom{n_1}{2}$ and $b_2(\mathfrak{f}_{n_1,2})=2\binom{n_1+1}{3}$.
\end{proof}
\section{Some computations in the $2$-step nilpotent case} \label{sec: examples}
To illustrate how our approach using free nilpotent extensions can be useful for doing explicit calculations, we will compute the second Betti number of two classes of $2$-step nilpotent Lie algebras. \\
Let us first agree upon a basis for $\mathfrak{f}_{n,3}$, the free $3$-step nilpotent basis on $\{x_1,\ldots,x_n\}$.
Recall that $\mathfrak{f}_{n,3}$ is graded: $\mathfrak{f}_{n,3}=H_{1}(n)\oplus H_{2}(n)\oplus H_{3}(n)$.
As basis for $H_1(n)$, we will obviously take $\{x_1,\ldots,x_n\}$. As basis for $H_2(n)$, we will take
\[ \{ [x_i,x_j] \ | \ i < j , i,j \in \{1,\ldots n \}\}. \]
We clearly have $|H_2(n)|=\binom{n}{2}$. Finally, as basis for $H_3(n)$ we will take $A \cup B \cup C$, with
\[A=\{ [x_i,[x_i,x_j]] \ | \ i,j\in \{1,\ldots,n \}  \ \mbox{s.t.} \ i<j\} \]
\[B=\{ [x_j,[x_i,x_j]] \ | \ i,j \in \{1,\ldots,n \} \ \mbox{s.t.} \ i<j\} \]
\[C=\{ [x_i,[x_j,x_k]],[x_j,[x_i,x_k]]  \ | \ i,j,k \in \{1,\ldots,n \} \ \mbox{s.t.} \ i<j<k \}. \]
One checks that $ |H_3(n)|=2\binom{n+1}{3}$.
\begin{example}\rm
Let $\mathfrak{g}$ be a $2$-step nilpotent Lie algebra over $k$ of type $(n,1)$, i.e. $|[\mathfrak{g},\mathfrak{g}]|=1$. It is known that in this case $\mathfrak{g}$ is isomorphic to $k^m\oplus \mathfrak{h}_n$ for some $m$ and $n$ such that $2n+m=n$, where $k^m$ is the $m$-dimensional abelian Lie algebra and $\mathfrak{h}_n$ is the $(2n+1)$-dimensional Heisenberg algebra (for example, see \cite{DecatDekimpeIgodt}, \cite{GozeKhakimdjanov}). Hence, the K\"{u}nneth formula entails that $b_2(\mathfrak{g})=b_2(\mathfrak{h_n})+\binom{m}{2}+2mn$. If $n=1$, then $\mathfrak{h}_n$ is just the free $2$-step nilpotent Lie algebra on $2$ generators, so $b_2(\mathfrak{h}_1)=2$. If $n\geq 2$, we have $b_2(\mathfrak{h}_n)= \binom{2n}{2}-1$. Note that the Betti numbers of $\mathfrak{h}_n$ over field of characteristic zero were computed in \cite{Santharoubane}, and over a field of prime characteristic in \cite{CairnsJambor}. \\
Now let us compute $b_2(\mathfrak{h}_n)$ ($n \geq 2$) using a free $2$-step nilpotent extension. Take $\mathfrak{f}_{2n,2}$ to be the free $2$-step nilpotent Lie algebra generated by $\{x_1,\ldots,x_n,y_1,\ldots,y_n\}$.  $\mathfrak{h}_n$ is a $2$-step nilpotent $(2n+1)$-dimensional Lie algebra with basis $\{x_1,\ldots,x_n,y_1,\ldots,y_n,z\}$ and non-zero Lie brackets given by $[x_i,y_i]=z$ for all $i \in \{1,\ldots,n\}$. So $\mathfrak{h}_n$ is of type $(2n,1)$. We leave it to the reader to check that $\mathfrak{h}_n$ has free $2$-step nilpotent extension
\[0 \rightarrow \mathfrak{n} \rightarrow \mathfrak{f}_{2n,2} \xrightarrow{\pi} \mathfrak{h}_n \rightarrow 0,\]
where $\mathfrak{n}$ is generated by the elements \[\{[x_i,x_j] \ , \  [y_i,y_j] \ , \ [x_i,y_j] \ , \ [x_i,y_i]-[x_j,y_j] \ | \ i \neq j \}.\]  We claim that $[i(\mathfrak{n}),\mathfrak{f}_{2n,3}] = \mathfrak{f}_{2n,3}^{3}$. To prove this, consider a basis for $\mathfrak{f}_{2n,3}^3=H_{3}(2n)$ as the one described above. Looking at the generators for $\mathfrak{n}$, the only non-trivial thing to verify is whether the basis elements of the form $[x_j,[x_i,y_i]]$ and $[y_j,[x_i,y_i]]$ are contained in $[i(\mathfrak{n}),\mathfrak{f}_{2n,3}]$. If $j$ does not equal $i$, we can use the Jacobi identity to see that these elements are contained in $[i(\mathfrak{n}),\mathfrak{f}_{2n,3}]$. Now suppose that $i=j$. We will show that $[x_i,[x_i,y_i]] \in [i(\mathfrak{n}),\mathfrak{f}_{2n,3}]$. Take some $s \neq i$. Because $[x_i,y_i]-[x_s,y_s]$ is a generator for $\mathfrak{n}$, we know that $[x_i,[x_i,y_i]]-[x_i,[x_s,y_s]]\in [i(\mathfrak{n}),\mathfrak{f}_{2n,3}]$. Since the second term is contained in $[i(\mathfrak{n}),\mathfrak{f}_{2n,3}]$, we are done. Similarly, we see that $[y_i,[x_i,y_i]] \in [i(\mathfrak{n}),\mathfrak{f}_{2n,3}]$, so the claim is proven. \\
Since $b_2(\mathfrak{f}_{2n,2})=| \mathfrak{f}_{2n,3}^3|=| [i(\mathfrak{n}),\mathfrak{f}_{2n,3}]|$ and $|\frac{\mathfrak{n}}{[\mathfrak{f}_{2n,2},\mathfrak{f}_{2n,2}]}|=|\mathfrak{n}|=\binom{2n}{2}-1$, Corollary \ref{cor: central dimension exprs} implies that $b_2(\mathfrak{h}_n)= \binom{2n}{2}-1$. Using Corollary \ref{cor: central extension}, this also shows that $\mathfrak{h}_n$ ($n \geq 2$) does not have central extensions of class $3$. \\
We can also use Corollary \ref{cor: central dimension exprs} to construct a basis for $\mathrm{H}^2(\mathfrak{h}_n,k)$. To do this, we first construct a $2$-cocycle that classifies our chosen free $2$-step nilpotent extension of $\mathfrak{h}_n$. Let $s: \mathfrak{h}_n \rightarrow \mathfrak{f}_{2n,2}$ be the vector space splitting of $\pi:\mathfrak{f}_{2n,2} \rightarrow \mathfrak{h}_n$ defined by $s(x_i)=x_i$, $s(y_i)=y_i$ for all $i \in \{1,\ldots,n\}$ and $s(z)=[x_1,y_1]$. Then $f: \Lambda^2(\mathfrak{h}_n) \rightarrow \mathfrak{n}: a \wedge b: [s(a),s(b)]-s([a,b])$ is a $2$-cocycle that classifies our chosen $2$-step nilpotent extension of $\mathfrak{h}_n$. One easily verifies that $f$ is determined by $f(x_i \wedge x_j)=[x_i,x_j]$, $f(y_i \wedge y_j)=[y_i,y_j]$ and $f(x_i \wedge y_i)=[x_i,y_i]-[x_1,y_1]$ for all $i,j$. Now choose a basis for $\mathrm{Hom}_k(\mathfrak{n},k)$ and denote it by $\{\varphi_1,\ldots,\varphi_{d}\}$. Then $\{\overline{\varphi_1 \circ f},\ldots, \overline{\varphi_{d} \circ f}\}$ is a basis for $\mathrm{H}^2(\mathfrak{h}_n,k)$.
\end{example}
\begin{example} \label{ex: 1dim kernel}\rm
Let $\mathfrak{g}$ be a $2$-step nilpotent Lie algebra over $k$ of type $(n,\binom{n}{2}-1)$ with free $2$-step nilpotent extension
\begin{equation} \label{eq: example} 0 \rightarrow \mathfrak{n} \rightarrow \mathfrak{f}_{n,2} \xrightarrow{\pi} \mathfrak{g} \rightarrow 0, \end{equation}
i.e. $|\mathfrak{n}|=1$. Let $X \in [\mathfrak{f}_{n,2},\mathfrak{f}_{n,2}] $ be a basis for $\mathfrak{n}$, hence $\mathfrak{n}=<X>$. We define the \emph{length} $l(X)$ of $X$ as follows
\[ l(X)=\min\{ k \in \mathbb{N}_0 \ | \ X = [a_1,b_1]+ \ldots + [a_k,b_k] \ \mbox{for some} \ a_i,b_i \in \mathfrak{f}_{n,2} \}. \]
Take $X_1,X_2 \in [\mathfrak{f}_{n,2},\mathfrak{f}_{n,2}]$. In \cite{DekimpeDeschamps}, it is proven that $\frac{\mathfrak{f}_{n,2}}{<X_1>}\cong \frac{\mathfrak{f}_{n,2}}{<X_2>}$ if and only if $l(X_1)=l(X_2)$. It is also shown that $l(X)\leq \lfloor \frac{n}{2} \rfloor$ for every $X \in [\mathfrak{f}_{n,2},\mathfrak{f}_{n,2}]$. It follows that there are, up to isomorphisms, exactly $\lfloor \frac{n}{2} \rfloor$ $2$-step nilpotent Lie algebras of type  $(n,\binom{n}{2}-1)$.
We will show that all these Lie algebras have the same second Betti number, namely $2\binom{n+1}{3}-n+1$. \\
Consider the free nilpotent extension $(\ref{eq: example})$, with $\mathfrak{n}=<X>$. Denote the generators of $\mathfrak{f}_{n,3}$ by $\{x_1,\ldots,x_{n}\}$ and give $\mathfrak{f}_{n,3}$ the basis described in the beginning of this section. We claim that $|[i(\mathfrak{n}),\mathfrak{f}_{n,3}]|=n$. Obviously, we can uniquely write $X = \sum_{i < j}a_{i,j}[x_i,x_j]$ and $[i(\mathfrak{n}),\mathfrak{f}_{n,3}]=\mathrm{span}\{[x_i,X] \ | \ i \in \{1,\ldots,n \}\}$. Hence, to prove the claim it suffices to show that $\{[x_i,X] \ | \ i \in \{1,\ldots,n \}\}$ is a linearly independent subset of $\mathfrak{f}_{n,3}^3$. Assume that $(\lambda_1,\ldots,\lambda_{n}) \in k^{n}$ such that
\begin{equation} \label{eq: sum}\sum_{l,i < j}\lambda_l a_{i,j}[x_l,[x_i,x_j]]= 0. \end{equation}
Fix an $l \in \{1,\ldots,n\}$. If $a_{l,j}$ is non-zero for some $j$ (this implies $l < j$), then the term $\lambda_la_{l,j}[x_l,[x_l,x_j]]$ appears in the left hand side of (\ref{eq: sum}). Also note that $[x_l,[x_l,x_j]]$ cannot appear with any other coefficient in (\ref{eq: sum}). So if we write the left hand side of (\ref{eq: sum}) in terms of the basis for $\mathfrak{f}_{n,3}^3$, the coefficient of $[x_l,[x_l,x_j]]$ will be $\lambda_la_{l,j}$. It follows that $\lambda_la_{l,j}=0$, hence $\lambda_l=0$. If $a_{j,l}$ is non-zero for some $j$, then a similar reasoning shows that $\lambda_l$ is also zero. \\
Now assume that  $a_{l,j}$ and $a_{j,l}$ are zero for all $j$. Because $\omega \neq 0$, we know that $a_{i,j}\neq 0$ for some $i < j$. Hence the term $\lambda_la_{i,j}[x_l,[x_i,x_j]]$ appears in (\ref{eq: sum}). Now look at the expansion of the left hand side of (\ref{eq: sum}) in term of the basis for $\mathfrak{f}_{n,3}^3$ and keep in mind that $a_{l,s}$ and $a_{s,l}$ are zero for all $s$. If $l<j$, then the coefficient of $[x_l,[x_i,x_j]]$ will be $\lambda_la_{i,j}$. If $j<l$, then the coefficient of $[x_j,[x_i,x_l]]$ will be $\lambda_la_{i,j}$.
So in both cases it follows that $\lambda_la_{i,j}=0$, therefore $\lambda_l=0$. \\
Because $b_2(\mathfrak{f}_{n,2})=2\binom{n+1}{3}$ and $|\frac{\mathfrak{n}}{[\mathfrak{f}_{2n,2}.\mathfrak{f}_{2n,2}]}|=1$, Corollary \ref{cor: central dimension exprs} says that $b_2(\mathfrak{g})=2\binom{n+1}{3}-n+1$. It also follow from Corollary \ref{cor: central extension} that $2$-step nilpotent Lie algebras of type $(n,\binom{n}{2}-1)$ have a central extension of class $3$.
\end{example}
\section*{Acknowledgements}
The author would like to thank Nansen Petrosyan for the helpful discussions that led up to this paper.

\end{document}